\documentclass[12pt]{article}
\usepackage{algorithm,algorithmicx}
\floatname{algorithm}{Procedure}
\usepackage{algpseudocode}
\usepackage{amsmath,amssymb,amsthm,mathrsfs}
\usepackage{graphicx}
\usepackage[margin=2cm]{geometry}
\usepackage[dvipsnames]{xcolor}

\usepackage{tikz}
\usetikzlibrary{matrix}

\usepackage{pgfplots}

\algblock{Input}{EndInput}
\algnotext{EndInput}
\algblock{Output}{EndOutput}
\algnotext{EndOutput}

\newcommand{\Desc}[2]{\State \makebox[5em][l]{#1}#2}

\makeatletter
\newcommand{\StatexIndent}[1][3]{%
  \setlength\@tempdima{\algorithmicindent}%
  \Statex\hskip\dimexpr#1\@tempdima\relax}
\makeatother

\newtheorem{theorem}{Theorem}
\newtheorem{lemma}{Lemma}

\newtheorem{conjecture}{Conjecture}

\def\eref#1{$(\ref{#1})$}
\def\sref#1{\S$\ref{#1}$}
\def\lref#1{Lemma~$\ref{#1}$}
\def\Lref#1{Line~$\ref{#1}$}
\def\tref#1{Theorem~$\ref{#1}$}
\def\Tref#1{Table~$\ref{#1}$}
\def\fref#1{Figure~$\ref{#1}$}
\def\aref#1{Procedure~$\ref{#1}$}
\def\cref#1{Conjecture~$\ref{#1}$}

\def\CC{\mathcal{C}}
\def\II{\mathcal{I}}
\def\SS{\mathcal{S}}
\def\RLS{\mathscr{R}}

\def\FF{\mathcal{F}}
\def\id{\iota}
\def\eps{\varepsilon}
\def\UU{\mathcal{U}}
\def\dfrac#1#2{\lower0.15ex\hbox{\large$\frac{#1}{#2}$}} 

\def\perm#1#2{\sigma_{#1,#2}}

\renewcommand{\geq}{\geqslant}
\renewcommand{\leq}{\leqslant}
\renewcommand{\ge}{\geqslant}
\renewcommand{\le}{\leqslant}

\DeclareMathOperator*{\Succ}{{\textsc{Succ}}}

\DeclareMathOperator*{\curt}{\tau}

\begin{document}

\title{Canonical labelling of Latin squares\\ in average-case polynomial time}

\author{
  Michael J. Gill\thanks{Research supported by an Australian Government Research Training Program (RTP) Scholarship.}, \ Adam Mammoliti \ and \ Ian M.\ Wanless\\
\small School of Mathematics\\[-0.75ex]
\small Monash University\\[-0.75ex]
\small Vic 3800, Australia\\[-0.75ex]
\small\tt \{michael.gill, adam.mammoliti, ian.wanless\}@monash.edu
}

\date{}

\maketitle

\begin{abstract}
A Latin square of order $n$ is an $n\times n$ matrix in which each row and column contains each of $n$ symbols exactly once. For $\eps>0$, we show that with high probability a uniformly random Latin square of order $n$ has no proper subsquare of order larger than $n^{1/2}\log^{1/2+\eps}n$. Using this fact we present a canonical labelling algorithm for Latin squares of order $n$ that runs in average time bounded by a polynomial in $n$.

The algorithm can be used to solve isomorphism problems for many combinatorial objects that can be encoded using Latin squares, including quasigroups, Steiner triple systems, Mendelsohn triple systems, $1$-factorisations, nets, affine planes and projective planes.
\end{abstract}


\section{Introduction}\label{s:intro}

A $k \times n$ \emph{Latin rectangle} is a $k\times n$ array with
entries from a set of size $n$, such that each symbol occurs exactly
once in each row and at most once in each column.  An $n\times n$
Latin rectangle is a \emph{Latin square}.  A Latin subrectangle is a
subarray that is a Latin rectangle.  A Latin subrectangle that is a
square is a \emph{Latin subsquare}.  The subsquare is \emph{proper} if
it has order at least 2 and is not the whole square.
A \emph{row cycle of length $k$} is a $2\times k$ Latin subrectangle
$R$ that is minimal with respect to containment. In other words, there
is no $k'<k$ such that $R$ contains a $2\times k'$ Latin subrectangle.

A finite \emph{quasigroup} $(Q,\circ)$ is a non-empty finite set $Q$
with a binary operation $\circ$ such that the multiplication table of
$Q$ is a Latin square. Throughout this paper we will consider
quasigroups and Latin squares to be interchangeable. The equation
$x\circ y = z$ between elements of a quasigroup denotes that the
entry in row $x$ and column $y$ is symbol $z$ in the corresponding
Latin square.

Two Latin squares $L_1$ and $L_2$ are \emph{isotopic} if $L_2$ can be
obtained from $L_1$ by applying a row permutation $\alpha$, a column
permutation $\beta$ and a symbol permutation $\gamma$ to $L_1$. For
such permutations the triple $(\alpha,\beta,\gamma)$ is an \emph{isotopism}
from $L_1$ to $L_2$.  If $(G,\circ)$ and $(H,*)$ are the quasigroups with
multiplication tables given by $L_1$ and $L_2$ respectively, then
$\alpha(x)*\beta(y) = \gamma(x\circ y)$ for all $x,y\in G$. If
$\alpha = \beta$, then the isotopism is an \emph{rrs-isotopism}.
If $\alpha = \beta = \gamma$, then the isotopism is an \emph{isomorphism}.
There are six \emph{conjugates} of a Latin square that are obtained
by permuting the three coordinates in its (row, column, symbol) triples.
A \emph{paratopism} is obtained by combining this operation with isotopism.
Isotopism, isomorphism and paratopism induce equivalence
relations and the equivalence classes are called \emph{isotopism classes},
\emph{isomorphism classes} and \emph{species}, respectively. Species
are also sometimes called main classes. The time complexity of testing whether
two Latin squares are in the same species is at most six (i.e.~a constant)
times the complexity of testing whether they are isotopic.

We will use $O(\cdot)$, $o(\cdot)$ and $\omega(\cdot)$ with
their standard asymptotic meanings (as $n\rightarrow\infty)$.
The first published results concerning the time complexity of
determining whether two Latin squares are isotopic were presented in a
conference paper by Miller \cite{Mil78}. Miller extended a result, which
he attributes to Tarjan, that identifying if two groups of order $n$
are isomorphic is computable in $n^{O(\log{n})}$ steps. He did so by
showing that a Latin square has a minimal generating set of
$O(\log{n})$ size.  Given two Latin squares $L$ and $L'$, his
algorithm is to find a generating set $A = \{a_1,\dots,a_m\}$ of size
$m\leq \log{n}$ in $L$, then for every set $B = \{b_1,\dots,b_m\}$ of
$m$ elements in $L'$, check to see if the map $a_i\mapsto b_i$ induces
an isomorphism. This solves Latin square/quasigroup isomorphism in
$n^{O(\log{n})}$ time. He also gave a polynomial time reduction of
isomorphism to isotopism, thereby giving the same complexity bound for
the latter problem. Recently, Dietrich and Wilson \cite{DW22}
showed that group isomorphism can be tested in polynomial time
for almost all orders. However, their techniques rely heavily
on features that are specific to groups. There is no reason to expect
analogous reasoning to apply to quasigroups.

Colbourn and Colbourn \cite{CC80} gave another algorithm for
quasigroup isomorphism by describing a restricted class of
$2^{O(\log^2{n})}$ relabellings for a quasigroup, from which a
canonical labelling can be selected. As each labelling can be computed
in polynomial time, this gives a time complexity of $2^{O(\log^2{n})}$
for their algorithm. They also noted that their algorithm runs in
polynomial time when the quasigroup contains no proper subquasigroups
of size greater than $c$, for some constant $c$. Colbourn and Colbourn
\cite{CC81} later proved that several generalised block designs are
\emph{graph isomorphism complete}, i.e.~can be used to model any
instance of graph isomorphism with at most polynomially more
work. Chattopadhyay, Tor\'{a}n and Wagner \cite{CTW13} proved that
graph isomorphism is not AC$^0$-reducible to quasigroup isomorphism
which implies that quasigroup isomorphism is strictly easier than
graph isomorphism unless both can be achieved in polynomial
time.  Recently, Levet \cite{Lev23} has extended these results to Latin
square isotopism, showing that graph isomorphism is not
AC$^0$-reducible to Latin square isotopism.  Gro\v sek and S\'ys
\cite{GS10} and Kotlar \cite{Kot12,Kot14} have observed empirically
that cycle structure can frequently be used to solve isotopism
problems for Latin squares.  In this work we give the first rigorous
proof of the effectiveness of such methods.

We present an algorithm that, given a Latin square $L$, produces a
labelling $(\alpha,\beta,\gamma)$ in such a way that the Latin square
$L'$ obtained by applying the isotopism $(\alpha,\beta,\gamma)$ to $L$
only depends on the isotopism class of $L$.  Such a labelling
$(\alpha,\beta,\gamma)$ is said to be \emph{canonical} and the Latin
square $L'$ is the \emph{(canonical) representative} of the isotopism
class of $L$. Our first main result is this:

\begin{theorem}\label{t:isotopytimeLS}
  Let $L$ be a random Latin square of order $n$.  Our algorithm
  determines the canonical labelling of $L$ in average-case $O(n^5)$
  time.  In the worst case, we can find the canonical labelling of $L$ in
  $O\big(n^{\log_2{n}-\log_2\ell(L)+3}\big)$ time, where $\ell(L)$
  denotes the length of the longest row cycle in $L$.
\end{theorem}

In the above theorem and throughout this paper, any reference to a
random object will implicitly assume that the underlying distribution
is the discrete uniform distribution.  One can determine whether two
Latin squares are isotopic by finding their canonical labellings and
checking if those are equal. It follows from the above result that we
can determine whether two independent random Latin squares are
isotopic in $O(n^5)$ average-case time and $O(n^{\log_2{n}+3})$ time
in the worst case. Our worst case performance is no better than
previous algorithms, but this is the first proof of polynomial-time
average-case performance.

In the process of proving \tref{t:isotopytimeLS} we show the following
result of independent interest. 

\begin{theorem}\label{t:largeSS}
  Let $\eps>0$. With probability $1-\exp(-\omega(n\log^2 n))$,
  a random $n\times n$ Latin square has no proper subsquare of
  order larger than $n^{1/2}\log^{1/2+\eps}n$. 
\end{theorem}

After submitting our paper we learnt that Divoux, Kelly, Kennedy and
Sidhu \cite{DKKS24} had recently obtained \tref{t:largeSS}, but with
$\eps=0$.  It was conjectured in \cite{MW99} (see also \cite{KSS21})
that a random Latin square almost surely has no proper subsquare of
order greater than 3.  \tref{t:largeSS} and \cite{DKKS24} represent
the first significant progress towards that conjecture, although there
is still much room for improvement.

\tref{t:largeSS} is proved in \sref{s:Latin}. Our main algorithm is
presented in \sref{s:canonicalLS} and then we prove
\tref{t:isotopytimeLS}.
The remainder of the paper discusses adaptations of our main algorithm
to solve isomorphism problems for classes of other combinatorial objects 
that can be encoded using Latin squares. We discuss two such classes in
some detail: Steiner triple systems in \sref{s:STS} and
$1$-factorisations of complete graphs in \sref{s:fact}. Then in
\sref{s:losotros} we discuss a number of other objects more briefly,
including nets, affine planes, projective planes and Mendelsohn triple
systems.  For all these cases, our worst-case complexity does not
improve on previous results. In several cases we make mild conjectures
about the structure of random objects in a class which, if true, would
show that our algorithm runs in average-case polynomial time.  We end
the paper with some concluding comments in \sref{s:con}. This will
include some discussion of the fact that we believe our algorithms are
actually faster than the bounds we prove for them.

\section{Large subsquares}\label{s:Latin}

In this section we prove \tref{t:largeSS}. Before doing so, we provide
a brief history of results related to subsquares in Latin squares. The
number of \emph{intercalates} (subsquares of order $2$) is of particular
interest.  A simple counting argument shows that a Latin square of
order $n$ can have no more than $\frac{1}{4}n^2(n-1)$ intercalates and
it is well known that the Cayley tables of elementary abelian 2-groups
are (up to isotopism) the only Latin squares that achieve this bound.

McKay and Wanless~\cite{MW99} showed that most Latin squares contain
many intercalates.  They conjectured that for any $\eps>0$ a random
Latin square will almost surely have between $(1-\eps)\mu_n$ and
$(1+\eps)\mu_n$ intercalates, where $\mu_n=n(n-1)/4$. A partial
confirmation of this conjecture was given in \cite{KS18} and the full
conjecture was finally proven by Kwan, Sah and Sawhney \cite{KSS21}
with the following.

\begin{theorem}\label{t:KSS21}
  A random Latin square contains $(1-o(1))n^2/4$ intercalates with
  probability $1-o(1)$.
\end{theorem}

\tref{t:KSS21} is part of a more general theory on
large deviations of the number of intercalates. 
Kwan, Sah, Sawhney and Simkin \cite{KSSS22}, found
the following bound on the number of Latin squares of
order $n$ avoiding any intercalates.

\begin{theorem}
  The number of order $n$ Latin squares with no intercalates is at least
  $\left(e^{-9/4}n-o(n)\right)^{n^2}$.
\end{theorem}

Much less is known about the distributions of subsquares larger than
intercalates.  In \cite{MW99} it is conjectured that the expected
number of subsquares of order $3$ in a random Latin square is $1/18$
and that most Latin squares do not contain subsquares of order greater
than or equal to $4$. Limited progress has been made for this
conjecture, apart from \cite{DKKS24} and \tref{t:largeSS}.

More progress has been made towards understanding extreme cases for
the number of Latin subsquares.  Let $\zeta(n,m)$ be the largest number
of subsquares of order $m$ achieved by any Latin square of order $n$.
For example, a result mentioned above is that
$\zeta(n,2) \leq \frac{1}{4}n^2(n-1)$. More generally, \cite{BCW14}
found that $\zeta(n,2^d)$ is maximised by the Cayley table of the
elementary abelian 2-group of order $n$.  It was shown by van
Rees~\cite{Ree90} that $\zeta(n,3)\leq \frac{1}{18}n^2(n-1)$, although
here there are many examples achieving equality \cite{KW15}.  In
\cite{BCW14} it was shown that
$\zeta(n,2)\geq\frac{1}{8}n^3-O(n^{2})$,
$\zeta(n,3)\geq\frac{1}{27}n^3-O(n^{5/2})$ and that
$\zeta(n,m)\leq O(n^{3+\lfloor\log_2(m/3)\rfloor})$ for any constant $m$.
This last result was misstated in \cite{BCW14} with $n$ in place of $m$.
It improved on earlier bounds of
$\zeta(n,m)\leq O(n^{\sqrt{2m}+2})$ from \cite{BVW10} and
$\zeta(n,m)\leq n^{\lceil \log_2 m\rceil +2}$ in \cite{BSW13}.

It was a long standing conjecture of Hilton that there exists a Latin
square of every order $>6$ that does not contain a proper subsquare of
any size. Hilton's conjecture has recently been proved \cite{AW24}.

The remainder of the section is devoted to proving \tref{t:largeSS}. 
We will use the following results (see e.g.~McKay and Wanless \cite{MW99}).

\begin{theorem}\label{t:uppbndLS}
  The number of Latin squares of order $n$ is at most
  \[\exp(n^2\log n-2n^2+O(n\log^2 n)).\]
\end{theorem}

\begin{theorem}\label{t:lowbndLR}
  The number of $k\times n$ Latin rectangles is at least
  \[
  \prod_{i=0}^{k-1}n!\Big(\frac{n-i}{n}\Big)^n
  =\frac{n!^k}{n^{kn}}\bigg(\frac{n!}{(n-k)!}\bigg)^n.
  \]
\end{theorem}

Let $R$ be a $k\times n$ Latin rectangle, with $k \leq n$. A
\emph{completion} of $R$ is a Latin square $L$ of order $n$ such that
$L$ contains $R$ in the first $k$ rows. We will use the
following result from \cite{MW99}.

\begin{theorem}\label{t:extensions}
  If $A,B$ are two $k\times n$ Latin rectangles then the ratio of the
  number of completions of $A$ versus the number of completions of
  $B$ is $\exp(O(n\log^2n))$.
\end{theorem}

We also require this special case of Theorem 4.7 of Godsil and McKay~\cite{GM90}.

\begin{theorem}\label{t:randLR}
  Let $k<n/5$. The probability that a random $k\times n$ Latin rectangle
  contains a specific $k\times k$ Latin subsquare is
  \[
  n^{-k^2}\exp\big(O(k^3/(n-3k+1))\big)
  \]
  as $n\rightarrow\infty$.
\end{theorem}

We can now prove \tref{t:largeSS}, for which we use the following definition.
Let $P_{k,n,s}$ denote the probability that a random
$k\times n$ Latin rectangle contains a subsquare of order $s$; in particular,
\tref{t:largeSS} asserts that $P_{n,n,k}=\exp(-\omega(n\log^2n))$ as
$n\rightarrow\infty$ whenever $k=k(n)\ge n^{1/2}\log^{1/2+\eps}n$.

\begin{proof}[Proof of \tref{t:largeSS}]
  Using the union bound over all $\binom nk<2^n=\exp(O(n))$ sets of $k$
  rows, together with \tref{t:extensions}, we can see that
  it suffices to show that
  $P_{k,n,k}=\exp(-\omega(n\log^2n))$ whenever
  $k\ge n^{1/2}\log^{1/2+\eps} n$.
  Let $\beta\rightarrow 0$ sufficiently slowly.

  First suppose that
  $k=\alpha n$ where $\beta\le \alpha\le 1/2$.
  To bound $X_{k,n}$, the number of $k\times n$ Latin rectangles with a
  subsquare of order $k$, we note that there are $\binom{n}{k}=\exp(O(n))$
  ways to choose $k$ columns in which to place the subsquare,
  $\binom{n}{k}=\exp(O(n))$ ways to choose $k$ symbols for the subsquare and
  at most $\exp(k^2\log k-2k^2+O(n\log^2n))$ choices for the subsquare
  once the columns and symbols are chosen, by \tref{t:uppbndLS}.  The
  number of choices for the rest of the rectangle is certainly not more than
  $(n-k)!^k$, since each of the $k$ rows has $(n-k)!$ ways to complete it
  if we ignore restrictions imposed by other rows. Hence, by Stirling's
  approximation,	
  \begin{align*}
    X_{k,n}&\le\exp(k^2\log k-2k^2+k(n-k)(\log(n-k)-1)+O(n\log^2n))\\
    &=\exp(\alpha^2n^2(\log n+\log\alpha)
    -2\alpha^2n^2+\alpha(1-\alpha)n^2(\log n+\log(1-\alpha)-1)+O(n\log^2n)).
  \end{align*}
  
  Now, by \tref{t:lowbndLR} the total number of $k\times n$ Latin
  rectangles is at least
  \begin{align*}
  \frac{n!^{\alpha n}}{n^{\alpha n^2}}\bigg(\frac{n!}{(n-\alpha n)!}\bigg)^{\!n}
  &=\exp(-n^2(\alpha-\log n+1)-(1-\alpha)n^2(\log n+\log(1-\alpha)-1)
  +O(n\log^2 n)).
  \end{align*}
  Hence $P_{k,n,k}\le\exp(\theta(\alpha)n^2+O(n\log^2n))$, where  
  \begin{align*}
    \theta(\alpha)&=\alpha^2\log\alpha
    -2\alpha^2+\alpha(1-\alpha)(\log(1-\alpha)-1)
    +\alpha+1+(1-\alpha)(\log(1-\alpha)-1)\\
    &=\alpha^2\log\alpha
    +(1-\alpha^2)\log(1-\alpha)
    +\alpha-\alpha^2.
  \end{align*}
  Since $\frac{d\theta}{d\alpha}=2\alpha(\log\alpha-\log(1-\alpha)-1)$ we see
  that $\theta(\alpha)$ is decreasing for $\alpha\in(0,e/(e+1))$ and
  increasing for $\alpha\in(e/(e+1),1)$. Also,
  $\lim_{\alpha\rightarrow0}\theta(\alpha)=0=\lim_{\alpha\rightarrow1}\theta(\alpha)$.
  It follows that $\theta(\alpha)<0$ for $\alpha\in(0,1)$, proving that
  $P_{k,n,k}=\exp(-\omega(n\log^2n))$ as required.
  
  It remains to consider the case where $n^{1/2}\log^{1/2+\eps}n\le k<\beta n$. In
  this case \tref{t:randLR} tells us that the probability of a specific
  subsquare of order $k$ is $\exp(-k^2\log n+O(\beta k^2))$. As in the
  previous case, there are $\exp(k^2\log k-2k^2+O(n\log^2n))$ possible
  subsquares. Hence
  \[P_{k,n,k}\le\exp\big((\log k-\log n-2+O(\beta))k^2+O(n\log^2n)\big).\]
  If $k\le n^{1/2}\log^2n$ then we use
  $\log k\le\frac12\log n+O(\log\log n)$ and otherwise we simply use
  $\log k\le\log n$. Either bound leads to
  $P_{k,n,k}\le\exp\big(-(\frac12+o(1))n\log^{2+2\eps}n\big)$, completing the
  proof.
\end{proof}

Cavenagh, Greenhill and Wanless \cite{CGW08} considered the longest
row cycle in a random Latin square of order $n$. They gave good
evidence to believe this cycle will have length at least $n/\log n$,
but were only able to show that the length is at least
$\sqrt{n}/9$. They did this by showing in \cite[Thm 4.9]{CGW08} that the
first two rows of a random Latin square contain more than $9\sqrt{n}$
cycles with probability $o(\exp(-\frac14{\sqrt n}))$.  We use their
results to squeeze out a very slightly stronger result.

\begin{theorem}\label{t:longcyc}
  Let $\eps>0$.
  With probability $1-o(\exp(-\frac14n^{1/2}))$, the first two rows of a random
  Latin square of order $n$ contain a row cycle of length at least
  $n^{1/2}\log^{1-\eps} n$.
\end{theorem}

\begin{proof}
  Let $P$ denote the probability that the lengths of the row cycles in the
  first two rows are $c_1,\dots,c_\kappa$. We may assume that
  $\kappa\le9\sqrt{n}$ and $c_1\ge c_2\ge\cdots\ge c_\kappa$. It suffices to
  prove the probability that $c_1<n^{1/2}\log^{1-\eps} n$ is
  $o(\exp(-\frac14n^{1/2}))$.
  By Corollary 4.5 of \cite{CGW08},
  \begin{equation}\label{e:Pbnd}
  P\le n^{1/3}\prod_{i=1}^\kappa\frac{2}{c_i}.
  \end{equation}
  Fix $\delta$ such that $0<\delta<1/2$.
  Let $\kappa' \leq \kappa$ be the maximum integer
  such that $c_{\kappa'}\ge 2n^{\delta}$.
  Note that $\sum_{i=1}^{\kappa'}c_i=n-O(n^{1/2+\delta})=n-o(n)$, which
  implies that 
  $\kappa'\ge(n-o(n))/c_1\ge n^{1/2}(1+o(1))\log^{\eps-1}n$,
  whenever $c_1<n^{1/2}\log^{1-\eps}n$.
  So $P\le n^{1/3-\delta\kappa'}=\exp(-\delta n^{1/2}(1+o(1))\log^\eps n)$
  for all $c_1,\ldots,c_\kappa$ with $c_1<n^{1/2}\log^{1-\eps}n$,
  by \eref{e:Pbnd}.  By the Hardy-Ramanujan Theorem \cite{HR18}, there
  are at most $\exp(O(\sqrt{n}))$ options for $c_1,\dots,c_\kappa$
  such that $c_1<n^{1/2}\log^{1-\eps}n$, from which the result follows
  by the union bound.
\end{proof}

\section{Canonical labelling of Latin squares}\label{s:canonicalLS}

In this section we prove \tref{t:isotopytimeLS}. We achieve this by
defining Procedures \ref{a:rowcyclelabel}--\ref{a:canonical} which
will be used to find the canonical labelling of a Latin square.  We
provide a detail description of how each of the procedures work and
how together they find the canonical labelling of a Latin square.  We
then show that the canonical labelling of a Latin square $L$ found by our
algorithm only depends on the isotopism class of $L$
(\tref{t:isoLsamecan}).  Finally we prove \tref{t:isotopytimeLS} at
the end of the section by proving the worst-case time complexity
(\tref{t:timetoextendpartiallabelling}) and the average-case
complexity (\tref{t:polycanonical}).

For any integer $n$, let $[n]:=\{1,\ldots, n\}$. Our Latin squares will all
have $[n]$ as their symbol set (although, of course, their subsquares may not).
Let $r_i$ and $r_j$ be two rows of some Latin square $L$. We define
$\perm{i}{j}=\perm{i}{j}(L)$ to be the permutation of $[n]$ such that
$\perm{i}{j}(i\circ c)=j\circ c$ for all $c \in [n]$. In other words,
$\perm{i}{j}$ is the permutation that maps $r_i$ to $r_j$. By the
Latin property $\perm{i}{j}$ is a derangement, meaning that it has no fixed
points. The cycles in the disjoint cycle decomposition of
$\perm{i}{j}$ correspond to the row cycles in $L$ in $r_i\cup r_j$.
For any permutation $\pi$ we define the \emph{cycle structure}
to be the (weakly) decreasing ordered list of cycle lengths
in the disjoint cycle decomposition of $\pi$. We also define $\ell_{i,j}(s)$
to be the length of the row cycle in $r_i\cup r_j$ containing the symbol $s$.

Let $\star$ be a symbol not contained in $[n]$. We extend the natural
total order on $[n]$ by assuming that $x < \star$ for all $x \in [n]$.
A \emph{partial permutation} of $[n]$ is a map $\pi:[n]\rightarrow
[n]\cup\{\star\}$ that is injective on the preimage of $[n]$.  For
$x\in[n]$, we say that $x$ is \emph{unlabelled} by $\pi$ if
$\pi(x)=\star$ and say $x$ is \emph{labelled} by $\pi$ otherwise. We
also use $\star$ to denote the constant partial permutation in which
all points are unlabelled.  The \emph{size} of a partial permutation
$\pi$, denoted $|\pi|$, is the number of elements of $[n]$ that are
labelled by $\pi$. A triple of partial permutations
$(\alpha,\beta,\gamma)$, each acting on $[n]$, is considered to be a
\emph{partial labelling} of a Latin square where $\alpha$ is applied
to the rows, $\beta$ is applied to the columns and $\gamma$ is applied
to the symbols. If such a partial labelling has no unlabelled row, no
unlabelled column, and no unlabelled symbol, then the partial
labelling is an isotopism. Given any partial labelling
$(\alpha,\beta,\gamma)$, we say that a row cycle is \emph{unlabelled}
if it hits some column $y$ such that $\beta(y) = \star$.  For a partial
labelling $(\alpha,\beta,\gamma)$, we let $L_{\alpha\beta\gamma}$ be
the array on the rows labelled by $\alpha$ and the columns labelled by
$\beta$ resulting from applying each of $(\alpha,\beta,\gamma)$ to the
Latin square $L$. We will abuse terminology slightly and refer to the
\emph{preimage of a $\star$} in $L_{\alpha\beta\gamma}$. By this we
will mean the unlabelled symbol in the entry of $L$ which was mapped
by $(\alpha,\beta,\gamma)$ to the entry of $L_{\alpha\beta\gamma}$
containing the $\star$ in question.

Let $w$ and $w'$ be two words of length $n$ over the alphabet
$[n]\cup\{\star\}$ and let $w[i]$ and $w'[j]$ denote the elements in position
$i$ and $j$ in $w$ and $w'$, respectively. Then $w$ is \emph{lexicographically
less than} $w'$ if there exists a position $i \in [n]$ such that $w[i] < w'[i]$
and $w[j] = w'[j]$ for all positions $j<i$.  Let $k \leq n$ and let $R$ and
$R'$ be two $k \times n$ arrays with entries taken from $[n] \cup \{\star\}$.
Let $r_t$ and $r'_t$ denote the $t$\textsuperscript{th} row of $R$ and $R'$
respectively. By reading the rows of the array from left to right, $r_t$ and
$r'_t$ are words of length $n$ obtained from the alphabet $[n]\cup\{\star\}$.
We consider $R$ to be lexicographically less than $R'$ if there exists an $i$
such that $r_i$ is lexicographically less than $r'_i$ and $r_j = r'_j$ for all
$j < i$.

\begin{figure}[h]
  \[
  \begin{array}{c|ccccc}
    &t+1&t+2&\cdots&t+\rho-1&t+\rho\\
    \hline
    i&t+1&t+2& \cdots&t+\rho -1& t + \rho\\
    j&t+2&t+3&\cdots&t+\rho& t+1
  \end{array}
  \]
  \caption{\label{fig:rowcycle}A row cycle in $r_i\cup r_j$ in standard form.}
\end{figure}

Let $\Gamma_{i,j}(L)$ denote the cycle structure of $\perm{i}{j}$ in $L$ where
$i<j$. A row cycle corresponding to a cycle of length $\rho$ in $\perm{i}{j}$
is \emph{contiguous} if its column labels
form a discrete interval $[t+1,t+\rho]$
and $i\circ(t+u+1)=j\circ(t+u)$ for $1\leq u\leq\rho-1$.
A contiguous row cycle is \emph{increasing} if
$i\circ(t+u)<i\circ(t+u+1)$ for $1\leq u\leq\rho-1$
and is in \emph{standard form} if $i \circ (t + u) = t + u$
for $1\leq u \leq \rho$. See \fref{fig:rowcycle} for an example of
a cycle in standard form.
Furthermore, $r_i\cup r_j$ is in \emph{standard form} if
\begin{itemize}
  \item each row cycle of $r_i \cup r_j$ is in standard form, 
  \item the row cycles in $r_i \cup r_j$ are sorted by length, with shorter
  cycles occurring in columns with higher indices.
  \item $\Gamma_{i,j}(L)$ is lexicographically maximum among the cycle structures of
  pairs of rows of $L$.
\end{itemize}
We say that a Latin square is \emph{reduced} if the symbols in the
first row and column are in increasing order.  For matrices induced by
partial labellings of Latin squares, we use the same definition with
an additional requirement that no cell in the first row or column may
contain $\star$.  Let $\RLS$ be the set of reduced Latin squares which
have $r_1 \cup r_2$ in standard form.

Let $L$ be a Latin square and let $\RLS(L)$ be the set generated in
the following way. We define $R_{\max}=R_{\max}(L)$ to be the set of
pairs $(i,j)$ such that $\Gamma_{i,j}(L)$ is lexicographically maximum
among the cycle structures of pairs of rows of $L$.  Let $\iota$
denote the identity permutation.  For each $(i,j)\in R_{\max}$ and for
each column permutation $\beta$ such that the row cycles in $r_i\cup r_j$
of $L_{\id\beta\id}$ are contiguous and in weakly decreasing
order of length, we let $\gamma$ be the permutation of the symbols of
$L$ such that $L_{\id\beta\gamma}$ has $r_i\cup r_j$ in standard
form. We then find the unique permutation $\alpha$ of rows of $L$ such
that $L_{\alpha\beta\gamma}$ is in reduced form, and add
$L_{\alpha\beta\gamma}$ to $\RLS(L)$. The next result follows
immediately from the above set up.

\begin{lemma}\label{lem:isoclass}
  If $\II(L)$ denotes the isotopism class of a Latin square $L$,
  then $\RLS(L) = \RLS \cap \II(L)$.
\end{lemma}

Below, in Procedures \ref{a:rowcyclelabel}--\ref{a:canonical}, we
describe an algorithm for finding a canonical representative of the
isotopism class of a given Latin square. The top level of the process
(\aref{a:canonical}) considers pairs $(i,j)$ from $R_{\max}$ in
turn. For each pair $(i,j)$ a number of branching steps
(\aref{a:branch}) are performed, each of which is followed by an
extension phase (\aref{a:extend}).  In a branching step we choose a
currently unlabelled cycle with the longest length and label it (via
\aref{a:rowcyclelabel}). The branching results from the different
labellings that the cycle can be given and (when there are multiple
unlabelled cycles sharing the longest length) the particular cycle
chosen.  In the extension phase we extend the current partial
labelling by labelling unlabelled row cycles (again via
\aref{a:rowcyclelabel}) until the array induced by the labelled
rows and columns is a Latin square. If some rows and columns remain
unlabelled then we recursively undertake a new branching step
followed by another extension phase. Whenever an extension phase ends
with a labelling of the whole Latin square, the new labelled Latin
square is checked to see if it is lexicographically minimal amongst
the Latin squares so far obtained. At the end of the process, the
labelled Latin square that is lexicographically minimal among those
obtained in this way is declared to be the canonical representative.

\begin{algorithm}[tb]
  \caption{Label a single row cycle of a Latin square.\label{a:rowcyclelabel}}
  \begin{algorithmic}[1]
    \Input
    \Desc{$L$}{A Latin square.}
    \Desc{$(i,j)$}{A pair of row indices.}
    \Desc{$(\alpha, \beta, \gamma)$}{A partial labelling of $L$. }
    \Desc{$s$}{A symbol in $L$ that is unlabelled by $\gamma$.}
    \EndInput
    \Output
    \Desc{$(\alpha', \beta', \gamma')$}{A partial labelling of $L$.}
    \EndOutput
    \Procedure{LabelRowCycle}{}
    \State $\sigma\gets s$. 
    \State $\lambda\gets P[\ell_{i,j}(s)]$.  
    \State $(\alpha',\beta',\gamma') \gets (\alpha,\beta,\gamma)$.
    \For{$k$ from $1$ to $\ell_{i,j}(s)$}
    \State $\curt\gets\curt+1$.
    \State $\gamma'(\sigma)\gets \lambda$.
    \State $T_\gamma[\curt]\gets\sigma$.
    \State Let $b$ be the column index such that $L[i,b]=\sigma$. \label{line:updatebeta}
    \State $\beta'(b)\gets \lambda$.
    \State $T_\beta[\curt]\gets b$.
    \State Let $a$ be the row index where $L[a,c_1]=\sigma$ and $\beta'(c_1)=1$. \label{line:updatealpha}
    \State $\alpha'(a)\gets \lambda$.
    \State $T_\alpha[\curt]\gets a$.
    \State $\sigma\gets L[j,b]$.
    \State $\lambda\gets \lambda+1$.
    \EndFor
    \State $P[\ell_{i,j}(s)] \gets P[\ell_{i,j}(s)] + \ell_{i,j}(s)$.
    \State \Return $(\alpha',\beta',\gamma')$.
    \EndProcedure    
  \end{algorithmic}
\end{algorithm}

The inputs for \aref{a:rowcyclelabel} are a Latin square $L$, a pair
$(i,j)$ of indices of rows in $L$, a partial labelling
$(\alpha,\beta,\gamma)$ of $L$ and a symbol $s$ unlabelled by
$\gamma$.  Let $R$ denote the row cycle in $r_i\cup r_j$ containing
$s$.  \aref{a:rowcyclelabel} produces a new partial labelling
$(\alpha',\beta',\gamma')$ in which $R$ maps to a row cycle in
standard form.  This is achieved by augmenting the partial labelling
$(\alpha,\beta,\gamma)$ without changing labels that have already been
assigned.  \aref{a:rowcyclelabel} has $\ell_{i,j}(s)$ steps, each of
which labels one new row, one new column and one new symbol. At the
beginning of each step, $\sigma$ is a symbol yet to be labelled and
$\lambda$ is a label yet to be assigned. Initially $\sigma$ is $s$ and
$\lambda$ is $P[\ell_{i,j}(s)]$. The value $P[\ell_{i,j}(s)]$ will be
chosen in such a way that the initial and subsequent values of
$\lambda$ are not labels in $\alpha$, $\beta$ or $\gamma$. The exact
mechanism for how this is achieved will be specified later.  Assume that there
is a column $c_1$ that satisfies $\beta'(c_1)=1$, a fact that will be justified
shortly.  \aref{a:rowcyclelabel} proceeds to
augment $\gamma'$ by setting $\gamma'(\sigma) =\lambda$, augment
$\beta' $ by setting $\beta'(b) =\lambda$, where $b$ is the column
index such that $L[i,b]=\sigma$ and augment $\alpha'$ by setting
$\alpha'(a)=\lambda$, where $a$ is the row index such that
$L[a,c_1]=\sigma$.  The steps ends by updating $\sigma=L[j,b]$ and
incrementing $\lambda$ by $1$.  These updates ensure that successive
values of $\sigma$ are given consecutive labels and that $R$ is mapped
to a row cycle in standard form in which $s$ is assigned the smallest
label. Note that while $\sigma$ and $b$ are given the same label in
the present algorithm, this will change in \sref{s:STS}, where we will
have a variant of \aref{a:rowcyclelabel} which does not put $R$ into
standard form.  The choice of $a$ ensures that the first row and
column of $L_{\alpha' \beta' \gamma'}$ are the same.
Note that while running \aref{a:canonical}, if we ever call
\aref{a:rowcyclelabel} with $(\alpha,\beta,\gamma)=(\star,\star,\star)$
then it will be the case that $P[\ell_{i,j}(s)]=1$. This ensures
that $c_1$ is defined before we reach \Lref{line:updatealpha}
and also that $\alpha'(i)=1$.
Subsequent calls to \aref{a:rowcyclelabel} which have
$(\alpha,\beta,\gamma)\ne(\star,\star,\star)$ will inherit the previous
value of~$c_1$. 

For each $(i,j)$ in $R_{\max}$, \aref{a:canonical} begins with the
empty partial labelling $(\star,\star,\star)$
and iteratively labels new elements using only
\aref{a:rowcyclelabel}, or reverts to a previously obtained partial
labelling via backtracking.  Therefore, at any stage of
\aref{a:canonical}, the current labelling $(\alpha,\beta,\gamma)$
satisfies the following properties, which are preserved by
the way that it employs \aref{a:rowcyclelabel}:
 \begin{itemize}
 \item each labelled row cycle in the first two rows of $L_{\alpha \beta \gamma}$
   is in standard form,
 \item $L_{\alpha\beta\gamma}$ is in reduced form, 
 \item $\alpha([n])=\beta([n])=\gamma([n])$.
 \end{itemize}

The following definitions and data structures are useful for
determining the efficiency of Procedures
\ref{a:rowcyclelabel}--\ref{a:canonical}.  For $k\in[n]$, let
$\mathcal{P}_k$ be the set of cycles in $r_i \cup r_j$ of length
strictly greater than $k$.  Let $P$ be a dynamic list of $n$ integers
with initial values $P[k]=1+\sum_{R\in\mathcal{P}_k}|R|$. 
This initial value of $P[k]$ will be the lowest index of any
column that occurs within a cycle
of length $k$ when $r_i \cup r_j$ is in standard form. 
It is clear
that the list $P$ can be computed efficiently while determining the
cycle structure of $r_i\cup r_j$.  Throughout \aref{a:canonical}, we
maintain a variable $\curt$ which is the number of symbols that have
been labelled thus far. We initialise $\curt :=0$ when no symbols have
been labelled.  Whenever \aref{a:rowcyclelabel} is
used to label a row cycle of length $k$ it will increase both
$P[k]$ and $\curt$ by $k$.
The definition of $\mathcal{P}_k$ and the initialisation of $P$ ensures
that at any stage $P[k]$ is at most the initial value of $P[k-1]$, with
equality only if all row cycles of length $k$ have been labelled. When
backtracking in \Lref{line:backtrack} of \aref{a:branch}, changes to $P$
and $\curt$ will be undone whenever the previous partial labelling is restored,
requiring at most a linear number of steps. At the start of
\aref{a:canonical} we find the lengths of all row cycles. This allows us to
precompute all values of $\ell_{i,j}(s)$, as well as the initial values for
$P$.
The columns in the discrete interval $[P[k],P[k-1])$, are pre-allocated for the
cycles of length $k$ in $r_i\cup r_j$. Then at any stage until all cycles of
length $k$ are labelled, $P[k]$ holds the smallest pre-allocated label that has
not yet been assigned to any column in a row cycle of length $k$.  When a row
cycle $R$ of length $k$ is labelled by \aref{a:rowcyclelabel}, the symbol $s$
that is sent as a parameter is given a label $s'$ and its column is labelled
with $P[k]$. The rest of $R$ is then labelled with consecutive symbols and
columns from those starting points.

Throughout \aref{a:canonical} we use lists named $T_\alpha$, $T_\beta$
and $T_\gamma$ to keep track of the order in which elements have been
labelled by the current partial labelling.  Given a partial labelling
$(\alpha,\beta,\gamma)$, \aref{a:rowcyclelabel} labels a row cycle
$R$ by augmenting each of $\gamma$, $\beta$ and $\alpha$ in turn, 
labelling one element at a time. Each time
$\phi\in\{\alpha,\beta,\gamma\}$ is augmented, $T_\phi$ is updated
accordingly.  Necessarily, the elements of $R$ are labelled in
increasing order of the labels of the elements.

\begin{algorithm}[tb]
  \caption{Extend a partial labelling until it labels an entire Latin subsquare. \label{a:extend}}
  \begin{algorithmic}[1]
    \Input
    \Desc{$L$}{A Latin square.}
    \Desc{$(i,j)$}{A pair of row indices.}
    \Desc{$(\alpha, \beta, \gamma)$}{A partial labelling of $L$.}
    \EndInput
    \Output
    \Desc{$(\alpha', \beta', \gamma')$}{A partial labelling of $L$.}
    \EndOutput
    \Procedure{Extend}{}
      \State $(r',c') \gets\big(1,\curt-\ell_{i,j}(T_\gamma[\curt])+1\big)$. 
      \State $(\alpha', \beta', \gamma') \gets (\alpha, \beta, \gamma)$.
      \While {$c' \leq \curt$} \label{line:whileuptot}
        \State $s \gets L[T_\alpha[r'], T_\beta[c']]$. 
        \If{$\gamma'(s)=\star$}
        \State $(\alpha', \beta', \gamma') \gets $ \Call{LabelRowCycle}{$L$, $(i,j)$, $(\alpha', \beta', \gamma')$, $s$}. \label{line:labelrowcycle}
        \EndIf
        \State $(r', c') \gets \Succ(r',c')$. \label{line:succ}
      \EndWhile
      \State \Return $(\alpha', \beta', \gamma')$.
    \EndProcedure
  \end{algorithmic}
\end{algorithm}

 \medskip

We define a successor function on pairs of positive integers by 
\[
\Succ(x,y) =
\begin{cases}
  (1, x+1), & \text{if } y \leq 1, \\
  (x+1,y), & \text{if } x+1 \leq y,\\
  (x, y-1), & \text{if } x\ge y>1.
\end{cases}
\]

At the start of \aref{a:extend}, the columns of
$L_{\alpha\beta\gamma}$ are those that are either in a (possibly
empty) Latin subsquare or the row cycle that has just been labelled by
the most recent branch step.  Necessarily, that row cycle has length
$\ell_{i,j}(T_\gamma[\curt])$ and the Latin subsquare has order
$\curt-\ell_{i,j}(T_\gamma[\curt])$. If the Latin subsquare has
positive order then it will have been searched by a previous instance
of \aref{a:extend}, and it will contain no $\star$.  We thus begin our
new search at $(r',c')=(1,\curt-\ell_{i,j}(T_\gamma[\curt])+1)$.  From that
position we proceed to update $(r',c')$, in the order induced by
$\Succ$, until we find a
cell $(T_\alpha[r'],T_\beta[c'])$ of $L_{\alpha\beta\gamma}$ that
contains a $\star$.  Let $s$ denote the unlabelled symbol that is the
preimage of this $\star$. Then the row cycle in $r_i\cup r_j$
containing $s$ is next labelled via \aref{a:rowcyclelabel}.  We
stress that the value of $\curt$ is increased by calling
\aref{a:rowcyclelabel}, which extends the search space governed by
\Lref{line:whileuptot} of \aref{a:extend}. This process of searching
and labelling is repeated until every cell in $L_{\alpha\beta\gamma}$
has been considered and no $\star$ remains.  At this point,
$L_{\alpha\beta\gamma}$ must be a Latin square, isotopic to a Latin
subsquare of $L$, and \aref{a:extend}
ends. Under certain very special circumstances \aref{a:extend} will be
called when $L_{\alpha\beta\gamma}$ is already a Latin square. In that
case, \aref{a:extend} will return the same partial labelling that was its input.
In all cases, \aref{a:extend} guarantees that the labelling it returns
is an isotopism of a Latin subsquare of $L$.

To maintain the required computational efficiency, our canonical labelling
process will keep track of a global dynamic version of the lists $P$,
$T_\alpha$, $T_\beta$ and $T_\gamma$, as well as $\curt$.
As each cycle in $r_i\cup r_j$ is labelled, we update these variables
to reflect the current status of the new labelling. 
For simplicity we omit most details of the bookkeeping needed to
maintain these lists. We note, however, that \aref{a:canonical} has
\Lref{line:updateP} to reflect that $P$ needs to be reinitialised
based on the cycle structure of $r_i \cup r_j$ when considering a new
pair $(r_i,r_j)$. 
All updating of each list requires at most linear
time with respect to the number of labelled elements in the current
partial labelling. Thus, at any point we may reference these lists
freely without adversely affecting the computational efficiency of the
algorithm.

We can now provide a more detailed outline of
the algorithm for finding a canonical labelling.  
Given a Latin square $L$ we select a row-column pair from
$R_{\max}$ and initialise an empty partial labelling.  We then proceed
by recursion. At the beginning of each stage in the recursion we have a partial
labelling $(\alpha,\beta,\gamma)$ such that
$L_{\alpha\beta\gamma}$ is a (possibly empty) Latin square of order
less than the order of $L$.  For each
symbol $s$ in a longest unlabelled row cycle of length $k$, we apply
\aref{a:rowcyclelabel} to assign $s$ the label $P[k]$. \aref{a:extend}
iteratively labels row cycles via \aref{a:rowcyclelabel}, where the
symbol chosen at each stage is the preimage of the $\star$ that
occupies the least cell with respect to the total ordering induced by
$\Succ$. This is continued until the current partial labelling
$(\alpha',\beta',\gamma')$ induces a Latin square.
If the resulting partial labelling is an isotopism then we check if
$L_{\alpha'\beta'\gamma'}$ is the lexicographically minimal Latin
square generated so far. If so, we select this square as our new
candidate, otherwise we discard it. If $(\alpha',\beta',\gamma')$ is
not an isotopism then we recurse, invoking a new branch step
(\Lref{line:newbranch} of \aref{a:branch}).  After our recursion
is complete we return the labelling that gave the
lexicographically minimal Latin square amongst all labellings
generated by the search. This procedure is run for every pair of rows
in $R_{\max}$, after which the lexicographically minimal labelling is
chosen. This labelling is considered to be our canonical labelling.

\begin{algorithm}[tb]
  \caption{Search all possible labelling extensions.\label{a:branch}}
  \begin{algorithmic}[1]
    \Input
    \Desc{$L$}{A Latin square.}
    \Desc{$(i,j)$}{A pair of row indices.}
    \Desc{$(\alpha, \beta, \gamma)$}{A partial labelling of $L$.}
    \EndInput
    \Output
    \Desc{$(\alpha^*\!, \beta^*\!, \gamma^*)$ }{A labelling of $L$.}
    \EndOutput
    \Procedure{Branch}{}
    \State Let $\CC$ be the set of longest row cycles in
    $r_i \cup r_j$ that are unlabelled by $(\alpha,\beta,\gamma)$.
    \State Let $S$ be the set of all symbols in the row cycles of $\CC$.
    \State {$(\alpha^*\!, \beta^*\!, \gamma^*)\gets(\star,\star,\star)$} 
    \For{$s \in S$}
    \State $(\alpha',\beta',\gamma')\gets$ \Call{LabelRowCycle}{$L$, $(i,j)$, $(\alpha, \beta, \gamma)$, $s$}. \label{line:rowoutput}
    \State $(\alpha', \beta', \gamma')\gets$ \Call{Extend}{$L$, $(i,j)$, $(\alpha', \beta', \gamma')$}.
    \label{line:extendoutput}
    \If{$|\alpha'| < n$}
    \State $(\alpha', \beta', \gamma') \gets$\Call{Branch}{$L$, $(i,j)$, $(\alpha', \beta', \gamma')$}. \label{line:newbranch}
    \EndIf
    \If{$(\alpha^*,\beta^*,\gamma^*)=(\star,\star,\star)$ or $L_{\alpha'\beta'\gamma'} < L_{\alpha^*\beta^*\gamma^*}$} \label{line:madeisotopism}
    \State $(\alpha^*,\beta^*, \gamma^*) \gets (\alpha', \beta', \gamma')$.
    \EndIf
    \State Restore previous values of $P$ and $\curt$. \label{line:backtrack}
    \EndFor
    \State \Return $(\alpha^*, \beta^*, \gamma^*)$.
    \EndProcedure
  \end{algorithmic}
\end{algorithm}

\begin{algorithm}[tb]
  \caption{Determine the canonical labelling of a Latin square\label{a:canonical}}
  \begin{algorithmic}[1]
    \Input
    \Desc{$L$}{A Latin square.}
    \EndInput
    \Output
    \Desc{$(\alpha, \beta, \gamma)$}{An isotopism of $L$.}
    \EndOutput
    \Procedure{Canonical}{}
    \State Find the length of every row cycle in $L$. \label{line:cyclen}
    \State $(\alpha,\beta,\gamma)\gets (\star,\star,\star)$.
    \For{$(i,j) \in R_{\max}(L)$}
    \State Update $P$ with the cycle structure of $r_i\cup r_j$. \label{line:updateP}
    \State $(\alpha^*,\beta^*,\gamma^*) \gets$ \Call{Branch}{$L$, $(i,j)$, $(\star,\star,\star)$}.
    \label{line:firstbranch}
    
    \If{$(\alpha,\beta,\gamma)=(\star,\star,\star)$ or $L_{\alpha^*\beta^*\gamma^*} < L_{\alpha\beta\gamma}$}
    \State $(\alpha,\beta,\gamma) \gets (\alpha^*,\beta^*,\gamma^*)$.
    \EndIf
    \EndFor
    \State \Return $(\alpha,\beta,\gamma)$.
    \EndProcedure
  \end{algorithmic}
\end{algorithm}

We can now prove that the canonical labelling of a Latin square $L$ found
by \aref{a:canonical} only depends on the isotopism class of $L$.

\begin{theorem}\label{t:isoLsamecan}
  Let $L_1$ and $L_2$ be two isotopic Latin squares and let $L_1'$ and
  $L_2'$ be the Latin squares obtained by running \aref{a:canonical}
  on $L_1$ and $L_2$, respectively. Then $L_1'=L_2'$.
\end{theorem}

\begin{proof}
Let $\phi=(\alpha,\beta,\gamma)$ be an isotopism from $L_1$ to
$L_2$.  Suppose $(i_1,j_1) \in R_{\max}(L_1)$ and let
$(i_2,j_2)=(\alpha(i_1),\alpha(j_1))$ be the corresponding pair in
$R_{\max}(L_2)$.  We consider the search tree $T_x$ of
$\textsc{Branch}(L_x,(i_x,j_x),(\star,\star,\star))$ for each $x\in\{1,2\}$.
Here the root is $(\star,\star,\star)$ and for any node, its children
are the partial labellings produced by \Lref{line:extendoutput}
of $\textsc{Branch}$.
The depth of a node $\phi'$ is the number of nested calls to
$\textsc{Branch}$ that it takes to obtain $\phi'$, including the
initial call $\textsc{Branch}(L_x,(i_x,j_x),(\star,\star,\star))$. The
root has depth $0$.  
The leaves of $T_x$ correspond to isotopes of
$L_x$, and $\textsc{Branch}(L_x,(i_x,j_x),(\star,\star,\star))$
returns the lexicographically least of these.  \aref{a:canonical}
returns the lexicographically least of all the isotopes obtained from
$\textsc{Branch}(L_x,(i_x,j_x),(\star,\star,\star))$ over
$(i_x,j_x)\in R_{\max}(L_x)$.  Thus it suffices to show
that for each leaf $(\alpha_1,\beta_1,\gamma_1)$ in $T_1$, there is a leaf
$(\alpha_2,\beta_2,\gamma_2)$ in $T_2$ such that
$(\alpha_1,\beta_1,\gamma_1)=(\alpha_2\alpha,\beta_2\beta,\gamma_2\gamma)
=(\alpha_2,\beta_2,\gamma_2)\phi$. Here and henceforth we compose maps
from right to left and compose triples of maps componentwise.
Also, we adopt the convention that $\pi(\star)=\star$ for any
partial permutation $\pi$, which will allow us to compose partial permutations.

Let $A_k$ and $B_k$ be the set of nodes of depth $k$ in $T_1$ and
$T_2$, respectively.  It suffices to prove for
all $k\ge0$ that for all $\phi_1 \in A_k $ there exists $\phi_2\in B_k$
such that $\phi_1 =\phi_2 \phi$.  We proceed by induction on $k$,
where the base case $k=0$ is trivially true as
$A_0=\{(\star,\star,\star)\}=B_0$.

Let $\phi_1\in A_k$ be a non-root node  and let $\phi_1' \in A_{k-1}$ 
be its parent.
By the induction hypothesis, there exists $\phi_2' \in B_{k-1}$
such that $\phi_1' =\phi_2' \phi$.
Let $s_1$  be the symbol whose execution of
Lines~\ref{line:rowoutput}--\ref{line:extendoutput} in 
$\textsc{Branch}(L_1,(i_1,j_1),\phi_1')$ results in $\phi_1$.
Let $s_2=\gamma(s_1)$ and note by construction that
the symbol $s_2$ is unlabelled by $\phi_2'$.
Since $\phi$ is an isotopism,
$s_1$ and $s_2$ are contained in row cycles of the same length in
rows $(i_1,j_1)$ of $L_1$ and rows $(i_2,j_2)$ of $L_2$, respectively.
Let $\phi_2$ be the result of executing
Lines~\ref{line:rowoutput}--\ref{line:extendoutput} in 
$\textsc{Branch}(L_2,(i_2,j_2),\phi_2')$ for $s_2$.
It suffices to show that $\phi_1 =\phi_2 \phi$. 

We consider the execution of
Lines~\ref{line:rowoutput}--\ref{line:extendoutput} in
$\textsc{Branch}(L_x,(i_x,j_x),\phi_x)$ for the symbol $s_x$ for
$x\in\{1,2\}$.  As the global list $P$ depends only on the cycle
structure of the rows being considered and the lengths of the row
cycles that have been labelled already, we can assume that $P$ begins
and remains the same for $x\in\{1,2\}$, as long as both processes
label cycles of the same length.  If $\phi_x^*$ for $x\in\{1,2\}$ are partial
labellings of $L_x$ such that $\phi_1^* =\phi_2^* \phi$, then for any
symbol $s_x'$ of $L_x$ in a row cycle of length $k$ such that
$s_2'=\gamma(s_1')$, the output $\phi_x^\dag$
of $\textsc{LabelRowCycle}(L_x,(i_x,j_x),\phi_x^*,s_x)$ will
satisfy $\phi_1^\dag =\phi_2^\dag \phi$. Thus, as $\textsc{Branch}$
augments $\phi_x'$ by a series of applications of
\aref{a:rowcyclelabel}, the result will follow if the sequence of symbols
chosen for each application of \aref{a:rowcyclelabel} correspond, with
respect to $\gamma$. We have shown that the initial choices $s_1$ and $s_2$
are corresponding and as
$(L_1)_{\alpha_1^*\beta_1^*\gamma_1^*}=(L_2)_{\alpha_2^*\beta_2^*\gamma_2^*}$
the $\star$'s in the $\Succ$-least cells of
$(L_1)_{\alpha_1^*\beta_1^*\gamma_1^*}$ and
$(L_2)_{\alpha_2^*\beta_2^*\gamma_2^*}$ must occur in the same cell,
with their corresponding preimages related by $\gamma$.
\end{proof}

Before proving \tref{t:isotopytimeLS}, we require two results on the
performance of our algorithm.

\begin{lemma}\label{lem:symdouble} 
  Let $L$, $(i,j)$ and $(\alpha, \beta, \gamma)$ be the parameters of a call to \textsc{Branch}
  whilst running \aref{a:canonical}. 
  Suppose that $|\alpha|=|\beta|=|\gamma|=k$ and let
  $s$ be a symbol in a longest unlabelled row cycle in $r_i\cup r_j$. 
  Then the partial labelling $(\alpha',\beta',\gamma')$ obtained from
  $(\alpha,\beta,\gamma)$ by executing Lines~$\ref{line:rowoutput}$ and $\ref{line:extendoutput}$ of 
  \textsc{Branch} for $s$ satisfies $|\alpha'|=|\beta'|=|\gamma'|\ge 2k$.
\end{lemma}

\begin{proof}
 Necessarily $L_{\alpha\beta\gamma}$ is a Latin subsquare contained in
 $L_{\alpha'\beta'\gamma'}$. Thus as the row cycle containing $s$ is
 outside $L_{\alpha\beta\gamma}$, a column containing $s$ in
 $L_{\alpha'\beta'\gamma'}$ must contain at least $k$ symbols that are
 not labelled by $\gamma$.
  As \aref{a:extend} extends
  a partial labelling until it is a Latin square,
  it follows that $|\gamma'|\ge2k$.
\end{proof}

Each time that we reach \Lref{line:madeisotopism} of \aref{a:branch},
we have created an isotopism which is worthy of consideration as to
whether it creates the canonical labelling.  These isotopisms
correspond to leaves in the search tree that we are implicitly
exploring, with branches of the tree created by calls to
\aref{a:branch}.  Our next result bounds how much time is spent in
\aref{a:rowcyclelabel} and \aref{a:extend} per isotopism that is
created.

\begin{lemma}\label{lem:polyextend}
  The time taken to run \aref{a:canonical} is $O(n^2)$ times the number
  of leaves in the implicit search tree.
\end{lemma}

\begin{proof}
Let $L$ be the Latin square input and $(\alpha,\beta,\gamma)$ be the
partial labelling output of \aref{a:extend}. Each cell in
$L_{\alpha\beta\gamma}$ is checked at most once to see if it contains
a $\star$. For each $\star$ found,
\aref{a:rowcyclelabel} labels one row cycle, which is
done in $O(n)$ time. As $L_{\alpha\beta\gamma}$ has at most $n^2$
cells and there are at most $n$ unlabelled symbols, the total time
taken is $O(n^2)$. Since
the process of extending the partial labelling is fixed for any choice of
partial labelling made by the branch procedure, the time required to find any
single isotopism is uniformly bounded by $O(n^2)$.
\end{proof}

We can now bound the worst case performance of our algorithm.

\begin{theorem}\label{t:timetoextendpartiallabelling}
  Let $L$ be a Latin square of order $n$ and let $\ell(L)$ be the
  length of the longest row cycle in $L$. \aref{a:canonical}
  determines the canonical labelling of $L$ in
  $O(n^{\log_2{n}-\log_2{\ell(L)} + 3})$ time.
\end{theorem}

\begin{proof}
  We consider the search tree implicit in \aref{a:canonical}.
  The root is the empty partial labelling and each leaf corresponds
  to an isotopism.
  Each non-leaf state $(\alpha,\beta,\gamma)$ in the search tree is
  extended, for each $s\in S$, to a new child state
  $(\alpha',\beta',\gamma')$ by invoking Lines~\ref{line:rowoutput}
  and \ref{line:extendoutput} of $\textsc{Branch}$. By
  \lref{lem:polyextend}, each leaf is treated in $O(n^2)$ time.
  Therefore, to estimate the complexity of \aref{a:canonical} it
  suffices to bound the number of leaves in the search tree.
  
  We consider the initial empty partial labelling of $L$ to have depth
  0 and let $(i,j) \in R_{\max}$.  By assumption, $r_i \cup r_j$
  contains at least one row cycle of length $\ell(L)$. Thus,
  \Lref{line:rowoutput} of $\textsc{Branch}$ labels a symbol and its
  row cycle which must be of length $\ell(L)$ and
  \Lref{line:extendoutput} of $\textsc{Branch}$ extends the resulting
  partial labelling to that of a Latin square, which necessarily
  must have order at least $\ell(L)$.
  By \lref{lem:symdouble} the child of any state has at least twice as many
  labelled symbols as the state itself.  It follows that at depth $k \geq 1$
  there are at least $2^{k-1}\ell(L)$ labelled symbols.
  Thus the search tree has depth $d \leq \log_2(n/\ell(L))+1$.
  Each state has $|S|\le n$ children, so there are $O(n^d)$ leaves.
  Each of them requires at most $O(n^2)$ steps to reach it.
  So the time taken to canonically label $L$ is bounded above by $O(n^{d+2})$.
The result follows.
\end{proof}

In \sref{s:con} we will discuss some examples for which our algorithm
will not run in polynomial time. However, we can do better in the average case:

\begin{theorem}\label{t:polycanonical}
  \aref{a:canonical} will find the canonical labelling of a random Latin
  square $L$ of order $n$ in average-case $O(n^5)$ time.
\end{theorem}

\begin{proof}
Fix $0<\eps<1/4$. Let $L$ be a random Latin square of order $n$
and let $\ell(L)$ be the length of a longest row cycle in $L$.  By
\tref{t:largeSS}, with high probability $L$ will not contain a proper
Latin subsquare of size larger than
$n^{1/2}(\log{n})^{1/2+\eps}$. Also, with high probability,
$\ell(L)\ge n^{1/2}(\log{n})^{1-\eps}$ by \tref{t:longcyc}, which
means that no row cycle of length $\ell(L)$ is contained in a proper
subsquare of $L$.  It follows that any isotopism found via
\aref{a:canonical} requires a single call to $\textsc{Branch}$. 
In other words, for any $(i,j) \in R_{\max}$, we will find that
$\textsc{Branch}(L,(i,j),(\star,\star,\star))$
returns an isotopism without executing Line~\ref{line:newbranch}.
There are at most $\binom{n}{2}$
pairs in $R_{\max}$ and for each of them there can be at most $n$
symbols that are in a longest row cycle within that pair of rows.  By
\lref{lem:polyextend}, it takes $O(n^2)$ time to find an isotopism
given a symbol in a longest row cycle.  Hence, with high probability,
the time taken to canonically label $L$ is $O(n^2 n n^2)= O(n^5)$.

To check that the bad cases do not affect our average time we note that,
by Theorems~\ref{t:longcyc} and~\ref{t:timetoextendpartiallabelling},
they contribute no more than
\[
O\big(\exp(-\dfrac14n^{1/2})n^{\log_2 n+3}\big)=O\big(\exp(-\dfrac14n^{1/2}+O(\log^2n))\big)=o(1)
\]
to the average time.
\end{proof}

\tref{t:isotopytimeLS} now follows directly from
Theorems~\ref{t:timetoextendpartiallabelling} and~\ref{t:polycanonical}.

\section{Steiner triple systems}\label{s:STS}

In this section we describe how to adapt the procedures defined in \sref{s:canonicalLS} to Steiner quasigroups.
A quasigroup $(Q,\circ)$ is \emph{idempotent} if $x\circ x = x$ for
all $x\in Q$ and is \emph{commutative} if $x\circ y = y\circ x$
for all $x,y\in Q$. It is semisymmetric if it satisfies the implications
\[
x\circ y= z\ \Longrightarrow\ y\circ z = x\ \Longrightarrow\ z\circ x =y,
\]
for $x,y,z \in Q$. A quasigroup that is commutative and semisymmetric
is called \emph{totally symmetric}. Semisymmetric idempotent quasigroups
are equivalent to Mendelsohn triple systems. Totally symmetric
idempotent quasigroups are called \emph{Steiner quasigroups}
and are equivalent to Steiner triple systems.

A \emph{Steiner triple system} is a pair $(V,B)$, where $V$ is a
finite set whose elements are called vertices and $B$ is a set of
$3$-subsets of $V$, whose elements are called blocks, such that every
pair of vertices is in exactly one block. A \emph{subsystem} of a
Steiner triple system $(V,B)$ is a Steiner triple system $(V',B')$
such that $V' \subseteq V$ and $B' \subseteq B$. Two Steiner triple
systems, $(V_1,B_1)$ and $(V_2,B_2)$, are isomorphic if there is a
mapping $\phi: V_1 \rightarrow V_2$ such that $\{x,y,z\}\in B_1$ if and
only if $\{\phi(x),\phi(y),\phi(z)\}\in B_2$.
Given a Steiner triple system $(V,B)$, we can define a Steiner
quasigroup with multiplication on the set $V$ as $x\circ x = x$ for
all $x \in V$ and $x\circ y = z$ for all $\{x,y,z\} \in B$.

Let $L$ be the Cayley table of a Steiner quasigroup and $(V,B)$ be the
corresponding Steiner triple system.
The \emph{cycle graph} $G_{a,b}$ is the graph with vertices $V
\setminus \{a,b,c\}$, where $\{a,b,c\} \in B$ and whose edges are
precisely the pairs $\{x,y\}$ such that either $\{a,x,y\} \in B$ or
$\{b,x,y\} \in B$. The graph $G_{a,b}$ is simple, by the Steiner
property, and is necessarily a disjoint union of cycles. Each row
cycle in rows $a$ and $b$ of $L$ corresponds to a cycle in $G_{a,b}$.
However, this relationship is two-to-one, as we now show.

\begin{lemma}\label{l:rowcycleSTS}
  In the Cayley table of a Steiner quasigroup $(V,\circ)$,
  consider two rows with indices $r$ and $r'$. Let $C,S \subseteq
  V\setminus\{r,r'\}$ be the columns and symbols, respectively, of a
  row cycle $R$ between row $r$ and row $r'$.  Then $C$ and $S$ are
  disjoint and furthermore there exists a distinct row cycle $R'$,
  whose columns are $S$ and whose symbols are $C$.
\end{lemma}

\begin{proof}
  The claim that there exists a row cycle $R'$ with columns $S$ and
  symbols $C$ is immediate from the total symmetry of the Steiner
  quasigroup and the existence of $R$.  Furthermore, if $C$ and $S$
  are disjoint, then it is immediate that $R$ and $R'$ are distinct.
  Hence, it suffices to show that $C$ and $S$ are disjoint.  Let
  $C=\{c_1,\ldots,c_{k}\}$ and $S=\{s_1,\ldots, s_{k}\}$, where $k$ is
  the length of $R$. All subscripts in this proof will be indexed
  modulo $k$.  Without loss of generality, let $R$ consist of the
  Latin triples $(r,c_i,s_i)$ and $(r',c_i,s_{i+1})$ for $i=1,\ldots,k$.
  Suppose for a contradiction that $s_i=c_j$ for some $i,j$. Then
  $s_{j}=r\circ c_j=r\circ s_i=c_{i}$ and hence
  $s_{i+1}=r'\circ c_i=r'\circ s_j= c_{j-1}$. It follows
  that $s_{i+x}=c_{j-x}$ for any integer $x$. If $j-i$ is even, then
  $s_y=c_y$ for $y=\frac{i+j}{2}$ because $i+\frac{j-i}{2}=y=j-\frac{j-i}{2}$.
  If $j-i$ is odd, then $s_y=c_{y-1}$ for
  $y=\frac{i+j+1}{2}$, since
  $i+\frac{j-i+1}{2}=y$ and $j-\frac{j-i+1}{2}=y-1$.
  In the former case, the elements of $(r,c_y,s_y)$ are not distinct and in
  the latter case the elements of $(r',c_{y-1},s_y)$ are not distinct.
  However, by construction the only triples in rows $r$ and $r'$ with
  repeated elements are $(r,r,r)$ and $(r',r',r')$. These have been
  excluded by the conditions of the Lemma. 
\end{proof}

It is necessary to exclude $r,r'$ from the columns and symbols sets of
the row cycle in \lref{l:rowcycleSTS}, as any row cycle between row
$r$ and row $r'$ with either $r$ or $r'$ as a column or symbol must 
be the following $3$-cycle with columns and symbols $\{r,r',s\}$,
  \[
  \begin{array}{c|ccc}
    &r&r'&s\\
    \hline
    r&r&s&r'\\
    r'&s&r'&r 
  \end{array}
  \]
where $r\circ r'=s$.
In fact, this row cycle is contained in an order $3$ subsquare whose
row, column and symbol sets are all $\{r,r',s\}$. For each $r,r'$, we
call the unique row cycle of length $3$ containing $r$ and $r'$ {\em
  singular}, while the remaining row cycles of rows $r$ and $r'$ are
{\em non-singular}.  \lref{l:rowcycleSTS} implies that for every
non-singular row cycle $R$ with columns $C$ and symbols $S$ there
exists a (unique) row cycle $R'$ with columns $S$ and symbols $C$;
such a pair $(R,R')$ is called an {\em inverse pair}.  Each inverse
pair in rows $a$ and $b$ corresponds to a single cycle in $G_{a,b}$.
Every non-singular row cycle in the Cayley table of a
Steiner quasigroup of order $n$ has length at most $(n-3)/2$.
   
For a Steiner triple system $(V,B)$, every cycle graph $G_{a,b}$ is
$2$-regular. Colbourn, Colbourn and Rosa \cite{CCR83} (see also \cite{CR91}) showed that every
$2$-regular graph is the cycle graph of some pair of elements in some
Steiner triple system:

\begin{theorem}
  For any $2$-regular graph $G=(V',E)$ on $n-3$ vertices, there is a Steiner
  triple system $(V,B)$ of order $n$ such that $G_{a,b} =G$ for some $a,b \in
  V\setminus V'$. 
\end{theorem}

A Steiner triple system $(V,B)$ is \emph{perfect} if $G_{a,b}$ is a
single cycle for all $a,b \in V$. There are only a handful of orders
$n$ for which a perfect Steiner triple system of order $n$ are known
to exist; perfect Steiner triple systems of order $n$ for $n=7,9,25,
33$ have been known for some time, see \cite{CR91}.
Perfect Steiner triple systems of order $n$ were then constructed for
$$n\in\{79, 139, 367, 811, 1531, 25771, 50923, 61339, 69991\}$$ in
\cite{GGM99} and for $n=135859$ in \cite{FGG07}. A countably
infinite perfect Steiner triple system was constructed in \cite{CW12}.

A $(k,k')$-configuration in a Steiner triple system is a set of $k'$
triples that covers at most $k$ points. It is known (see
\cite[Thm~1.1]{FGG07}) that any Steiner triple system of order $n\geq k+3$
has a $(k+3,k)$-configuration. This has motivated interest in
Steiner triple system avoiding $(k+2,k)$-configurations. A
$(3,1)$-configuration is a single triple, while a $(4,2)$-configuration 
would require two triples to share two points, which is impossible
in a Steiner triple system. So it is only of interest to consider
$(k+2,k)$-configurations for $k \geq 3$. 
Any nonsingular row cycle $R$ of length $k$ forms a
$(2k+2,2k)$-configuration.  Each of the $2k$ cells of $R$ corresponds
to a different triple and between them these triples cover $2k+2$
points corresponding to $k$ columns, $k$ symbols and two rows.  Not
every $(2k+2,2k)$-configuration arises from a row cycle of length $k$
in this way. Erd\H{o}s conjectured the following, which has recently
been proven in \cite{KSSS22a}.

\begin{theorem}
  For any fixed integer $g$, there exists an integer $N$ such that for
  any integer $n\geq N$ satisfying $n \equiv 1,3\pmod 6$ there exists a
  Steiner triple system of order $n$ with no $(k+2,k)$-configuration
  for $2\leq k\leq g$.
\end{theorem}

\subsection{Canonical labelling of Steiner triple systems}

We could apply the algorithm developed in \sref{s:canonicalLS},
without change, to determine a canonical labelling of the Cayley table
of a Steiner quasigroup. This could be viewed as some sort of canonical
form of the corresponding Steiner triple system.  However,
the result might not be idempotent or totally symmetric, in which
case it is certainly not a representative of the isomorphism class
of the initial Steiner quasigroup.
To avoid this undesirable situation, in this subsection we propose two
possible canonical labellings by isomorphisms
$(\alpha,\beta,\gamma)=(\alpha,\alpha,\alpha)$. We simply refer to
such a labelling by the single permutation $\alpha$, and refer to the
Latin square $L_{\alpha\alpha\alpha}$ simply as $L_\alpha$. Partial
labellings are treated in the analogous way. Note that isomorphisms
preserve idempotency and total symmetry, and hence map Steiner
quasigroups to Steiner quasigroups.

Our first canonical labelling technique involves ``lifting'' an
isomorphism from the isotopism produced by our canonical labelling
algorithm in the previous section. It exploits an observation from
\cite{WI05} that isotopic idempotent symmetric Latin squares are
isomorphic.

\begin{theorem}\label{t:lift}
  Let $L$ be a Steiner quasigroup and let $(\alpha,\beta,\gamma)$ be
  the canonical labelling of $L$ obtained by \aref{a:canonical}. Then
  $L_\gamma$ is a Steiner quasigroup that serves as a canonical
  representative of the isomorphism class of $L$.
\end{theorem}

\begin{proof}
  Suppose that $L$ and $L'$ are isotopic Steiner quasigroups and let
  $L_{\alpha\beta\gamma}$ and $L'_{\alpha'\beta'\gamma'}$ be their
  respective canonical forms obtained from \aref{a:canonical}. By
  \tref{t:isoLsamecan}, these forms must be equal and hence
  $L=L'_{(\alpha'\alpha^{-1})(\beta'\beta^{-1})(\gamma'\gamma^{-1})}$.
  By \cite[Lem.\,6]{WI05} it follows that $L=L'_{(\gamma'\gamma^{-1})}$
  and hence $L_{\gamma}=L_{\gamma'}$. The result follows.
\end{proof}

Of course, given total symmetry, it would be equally valid to chose $L_\alpha$
or $L_\beta$ to be the canonical representative.

\subsubsection{Canonical Steiner subsystems}

The techniques in \sref{s:canonicalLS} determined canonical labellings
of Latin squares by labelling Latin subsquares and then branching if
the whole square had not been labelled. In the Cayley table of a
Steiner quasigroup there may be Latin subsquares which do not
correspond to Steiner subsystems. Such subsquares should not pose
obstacles to fast labelling, although they might if we used the
technique behind \tref{t:lift}. In this subsection we describe a
labelling algorithm that determines a canonical isomorphism of a
Steiner quasigroup in quasipolynomial time. This algorithm is similar
to the algorithm in \sref{s:canonicalLS} in that we will label row
cycles in a particular order. However, we will maintain the idempotent
totally symmetric structure of the Steiner quasigroup. The only
subsquares which might cause us to branch are those that correspond to
Steiner subsystems.

We first expand on some notions related to Steiner quasigroups from
earlier in the section.
Let $L$ be the Cayley table of a Steiner quasigroup.  Let $(R,R')$ be
an inverse pair in $L$.  The inverse pair $(R, R')$ is
\emph{contiguous} if $R$ is contiguous and the columns of $R\cup R'$
are an interval in $[n]$. Given the inverse pair $(R,R')$ we say
that $R$ is the \emph{left} row cycle and $R'$ is the \emph{right}
row cycle if every column in $R$ has a lower index than every column
in $R'$.

An inverse pair is \emph{paired increasing} if it is contiguous and the left
row cycle is increasing. Due to the total symmetry of the Steiner quasigroup
the remaining cycle is uniquely determined, see \fref{f:inversepairlabelled}.
We say that a Steiner quasigroup is in \emph{paired standard form} if
\begin{itemize}
  \item every inverse pair within the first two rows is paired increasing,
  \item the singular row cycle in the first two rows
    has $\{1,2,3\}$ as its symbols,
  \item the non-singular row cycles in the first two rows are sorted
    by length, with shorter cycles in columns with higher indices,
  \item $\Gamma_{1,2}(L)$ is lexicographically maximum among
    the cycle structures of all pairs of rows.
\end{itemize}

\begin{figure}[h]
  \[
  \begin{array}{c|cccc|cccc}
    &t&{t+1}&\cdots&{t+\rho-1}&{t+\rho}&{t+\rho+1}&\cdots&{t+2\rho-1}\\
    \hline
    1&t+\rho&t+\rho+1& \cdots&t+2\rho -1& t & t+1 & \cdots & t+\rho - 1 \\
    2&t+\rho+1&t+\rho+2&\cdots&t+\rho& t+\rho - 1 & t & \cdots & t+\rho-2 
  \end{array}
  \]
  \caption{A contiguous inverse pair that is labelled to be paired increasing.\label{f:inversepairlabelled}}
\end{figure}

Let $\SS$ be the set of all Steiner quasigroups
that are in paired standard form.  Let $L$ be a Steiner
quasigroup and let $\SS(L)$ be the set generated by the following
procedure. Let $R_{\max}=R_{\max}(L)$ be the set of all ordered pairs
$(i,j)$ such that $ \Gamma_{i,j}(L)$ is lexicographically maximum.
For each $(i,j) \in R_{\max}$,
we compute the partial permutation $\alpha$ such
that $\alpha(i) = 1$, $\alpha(j)=2$, $\alpha(i\circ j) = 3$ and
$\alpha(x)=\star$ otherwise. This creates
the singular row cycle $(132)$. We then extend the partial permutation $\alpha$
in all possible ways such that the resulting labelled quasigroup 
is in paired standard form. The next result is analogous
to \lref{lem:isoclass} and follows immediately from the definitions of
$\SS$ and $\SS(L)$.

\begin{lemma}
  If $\II(L)$ denotes the isomorphism class of the Cayley table $L$ of
  a Steiner quasigroup, then $\SS(L) = \SS\cap \II(L)$.
\end{lemma}

In the remainder of this section we will describe a
canonical labelling algorithm by modifying the algorithm given in
\sref{s:canonicalLS}. There are two main differences between the 
algorithm from \sref{s:canonicalLS} and the one that we will define
here. The first difference is the need to generate an isomorphism $\alpha$, 
rather than an isotopism $(\alpha,\beta,\gamma)$. The
second difference is that we may exploit the symmetry properties which allow us
to use the labelling of a non-singular cycle to immediately give the labelling
of the inverse pair containing that cycle, by \lref{l:rowcycleSTS}.

\begin{algorithm}[tb]
  \caption{Label an inverse pair of row cycles in the Cayley table of a Steiner quasigroup.
  \label{a:rowcyclelabelsts}}
  \begin{algorithmic}[1]
    \Input
    \Desc{$S$}{A Steiner quasigroup.}
    \Desc{$(i,j)$}{A pair of row indices.}
    \Desc{$\alpha$}{A partial labelling of $S$. }
    \Desc{$s$}{A symbol in $S$ that is unlabelled by $\alpha$.}
    \EndInput
    \Output
    \Desc{$\alpha'$}{A partial labelling of $S$.}
    \EndOutput
\Procedure{LabelInversePair}{}
\State $\sigma\gets s$.
\State $\lambda\gets P[\ell_{i,j}(s)]$.
\State $\alpha' \gets \alpha$.
\For{$k$ from $1$ to $\ell_{i,j}(s)$}
\State $\curt\gets\curt+2$.
\State $\alpha'(\sigma)\gets \lambda + \ell_{i,j}(s)$.
\State $T_\alpha[\curt-1]\gets\sigma$.
\State Let $b$ be the column such that $L[i,b]=\sigma$.
\State $\alpha'(b)\gets \lambda$.
\State $T_\alpha[\curt]\gets b$.
\State $\sigma\gets L[j,b]$.
\State $\lambda\gets \lambda+1$.
\EndFor
\State $P[\ell_{i,j}(s)] \gets P[\ell_{i,j}(s)] + 2\ell_{i,j}(s)$.
\State \Return $\alpha'$.
\EndProcedure    
\end{algorithmic}
\end{algorithm}
Most of the procedures used in our first algorithm can be repurposed
with minimal changes. The most substantial change is to
\aref{a:rowcyclelabel}, which requires care to preserve the
total symmetry of the Steiner quasigroup.  To that end, we present
\aref{a:rowcyclelabelsts} which will label an inverse pair. The
labelling obtained from \aref{a:rowcyclelabelsts} is consistent with
the labelling shown in \fref{f:inversepairlabelled}.
\aref{a:rowcyclelabelsts} is given a symbol $s$ in a row cycle $R$,
which is part of some inverse pair,
$(R,R')$. \aref{a:rowcyclelabelsts} produces a partial labelling
$\alpha'$ such that in $L_{\alpha'}$ we have that
$R$ is the left row cycle of
$(R,R')$ and $\alpha'(s)$ is the leftmost symbol in the first row
among all the symbols of $R\cup R'$.  Note that unlike
\aref{a:rowcyclelabel}, $\alpha'(s)$ is not the smallest label among
the symbols in $R\cup R' $.  \aref{a:rowcyclelabelsts} has
$\ell_{i,j}(s)$ steps, each one of which labels
one symbol of $R$ and one symbol of $R'$. At the beginning of each step,
$\sigma$ is the next symbol of $R$ to be labelled and $\lambda$ is the
smallest label yet to be given. \aref{a:rowcyclelabelsts} then
proceeds to give the symbol $\sigma$ the label $\lambda+\ell_{i,j}(s)$
and give the label $\lambda$ to the column $b$
such that $i\circ b=\sigma$. From \lref{l:rowcycleSTS} and the fact
that $\sigma$ is initially $s$, we see that
$\sigma$ is always a symbol in $R$ and
$b$ is always a symbol in $R'$. Finally, $\sigma$ and $\lambda$ are
updated by choosing $\sigma = L[j,b]$ and increasing $\lambda$ by 1,
ensuring that successive values of $\sigma$ in $R$
are given consecutive labels.
At each step, elements $\sigma$ and $b$
such that $i\circ b =\sigma$ are given labels $\ell_{i,j}(s)$ apart,
which ensures that $L_{\alpha'}[1,t] = t+\ell_{i,j}(s)$
for $P[\ell_{i,j}(s)]\le t<P[\ell_{i,j}(s)]+\ell_{i,j}(s)$,
since $\lambda$ is initially given the value
$P[\ell_{i,j}(s)]$.  The choice of $\sigma$ at each step ensures that
$L_{\alpha'}[2,t]=L_{\alpha'}[1,t+1]$ for 
$P[\ell_{i,j}(s)]\le t\le P[\ell_{i,j}(s)]+ \ell_{i,j}(s)-2$. Thus
the row cycle $R$ in $L_{\alpha'}$ is completely determined and is
increasing. It follows from total symmetry that $R'$ is determined
uniquely and $(R,R')$ is paired increasing.

Only minor changes to \aref{a:extend} are required.
First, replace each isotopism with an
isomorphism, i.e., set $\alpha=\beta=\gamma$ and
$\alpha'=\beta'=\gamma'$.
Note that \Lref{line:labelrowcycle} of \aref{a:extend}  now requires a
call to \aref{a:rowcyclelabelsts} and should be replaced with
$\alpha'\gets\textsc{LabelInversePair}(L, (i,j), \alpha', s)$.
We finish with the changes to \aref{a:extend} by noting that for a
partial labelling $\alpha$, it is unnecessary to check every
unexplored cell of $L_\alpha$ to ensure it contains no $\star$.
By the total symmetry of the Steiner quasigroup, we need only
check unexplored cells in, say, the upper triangular section of
$L_{\alpha}$ and so the use of $\Succ$ in \Lref{line:succ} of
\aref{a:extend} is inefficient. For this reason, we define a new successor
function that can be used instead of the successor function
defined earlier.
\[
\Succ{}'(x,y) =
\begin{cases}
  (1, y+1), & \text{if } x = y, \\
  (x+1,y), & \text{if } x < y.
\end{cases}
\]

Analogously to the changes made to \aref{a:extend}, each isotopism in
Procedures~\ref{a:branch} and \ref{a:canonical} is replaced with an
appropriate isomorphism and the call to \aref{a:rowcyclelabel} on
\Lref{line:rowoutput} of \aref{a:branch} is replaced with
$\alpha'\gets \textsc{LabelInversePair}(L, (i,j), \alpha, s)$.

Up to this point we have only been concerned with the labelling of
non-singular row cycles. To complete the labelling of the Steiner
quasigroup we need to consider the singular row cycle. Consider
some $(i,j) \in R_{\text{max}}(L)$ and let $i \circ j = k$.
It follows that $(ikj)$ is the $3$ cycle containing the idempotent
entries in rows $i$ and $j$. We define the partial labelling
$\alpha(i)=1$, $\alpha(j)=2$ and $\alpha(k)=3$. To account for
this change, the array $P$ requires modification. However, this
modification is trivial, as we calculate the values of $P$ by ignoring
the singular cycle and incrementing the value $P[k]$ by $3$,
for all $k \in [n]$. After labelling the singular cycle we continue by
updating $P$, $T_\alpha$ and $\curt$ appropriately. This must be done
for each new $(i,j)$ selected in \aref{a:canonical} before the first
branch step on \Lref{line:firstbranch}.

Having modified \aref{a:canonical} to use \aref{a:rowcyclelabelsts} to
preserve total symmetry, we have the follow analogue of
\tref{t:isoLsamecan} for Steiner quasigroups.

\begin{theorem}\label{t:isoSsamecan}
  Let $L_1$ and $L_2$ be Cayley tables of two isomorphic Steiner
  quasigroups and let $L_1'$ and $L_2'$ be the Latin squares obtained
  by running the modified \aref{a:canonical} on $L_1$ and $L_2$,
  respectively. Then $L_1'=L_2'$.
\end{theorem}

Finally, we observe that the changes made to the algorithm in
\sref{s:canonicalLS} do not change the worst case time
complexity of any of the procedures.  Thus we have:

\begin{theorem}\label{t:stscanonical}
  Let $L$ be a Steiner quasigroup and let $\ell(L)$ be the length of the
  longest row cycle in $L$. The modified \aref{a:canonical} determines the
  canonical labelling of $L$ in $O(n^{\log_2{n}-\log_2{\ell(L)}+ 3})$ time.
\end{theorem}

It is harder to establish the average time complexity of canonical
labelling for Steiner quasigroups. In particular, an analogue of
\tref{t:polycanonical} for Steiner quasigroups, which would require
analogues of Theorems~\ref{t:largeSS} and~\ref{t:longcyc}, appears
elusive.  However, we certainly believe it to be true:

\begin{conjecture}\label{con:steiner}
  Let $S$ be a Steiner quasigroup of order $n$ selected uniformly at
  random.  With probability $1-o(1)$, the length $\ell(S)$ of the
  longest row cycle in $S$ exceeds the size of the largest
  proper Steiner subsystem of $S$.
\end{conjecture}

Quackenbush \cite{Qua80} conjectures that random Steiner triple
systems have a low probability of having any proper Steiner
subsystem. However, Kwan \cite{Kwa20} suggests that this conjecture
fails with regard to Steiner subsystems of order $7$. However, we
believe that it will be true for subsystems of size larger than some
constant. If so then \cref{con:steiner} should be true with room to spare.
If it is, then our algorithm has polynomial average-time complexity.

\begin{theorem}
  If \cref{con:steiner} is true, then canonical labelling (and hence
  also isomorphism testing) of Steiner triple systems can be achieved
  in average-case $O(n^5)$ time.
\end{theorem}

\section{One-factorisations}\label{s:fact}

In this section we describe how the algorithm defined in
\sref{s:canonicalLS} can be adapted for $1$-factorisations of complete
graphs.
A $1$-factorisation of a graph is a partition of the edge set of that
graph into spanning $1$-regular subgraphs called $1$-factors. A
\emph{sub-$1$-factorisation} of order $k$ of a $1$-factorisation
$\{F_1,\ldots,F_{2n-1}\}$ of the complete graph $K_{2n}$ is a
$1$-factorisation $\{F_1',\ldots,F_{k-1}'\}$ of $K_{k}$ such that for
all $i=1,\ldots, k-1$, we have $F_i'\subseteq F_j$ for some $j$. Let
$F$ and $G$ be two $1$-factorisations of $K_{2n}$. Then $F$ and $G$
are isomorphic if there is a mapping
$\phi:V(K_{2n})\rightarrow V(K_{2n})$ such that $f \in F$ if and only
if $\phi(f) \in G$.

A quasigroup $(Q,\circ)$ is \emph{unipotent} if there is a element $y$
such that $x\circ x = y$ for all $x\in Q$. One can encode a
$1$-factorisation of $K_{2n}$ either as a unipotent commutative 
quasigroup of order $2n$ or an idempotent commutative quasigroup of
order $2n-1$.
In the former case, this is done for a $1$-factorisation
$\{F_1,\ldots,F_{2n-1}\}$ of $K_{2n}$ by defining multiplication
$\circ$ on the set $V(K_{2n})$ as $i\circ i = 1$ for all
$i\in V(K_{2n})$ and $x\circ y=k+1$ for every edge $\{x,y\}$ that
is in $F_k$.
See \cite{WI05} for further details.  It is also true that
$1$-factorisations of the complete bipartite graph are equivalent to
general Latin squares. Thus, isomorphism testing of $1$-factorisations of 
complete bipartite graphs can be performed in average-case polynomial
time, by \tref{t:isotopytimeLS}.

Before presenting an analogue to the algorithm
in \sref{s:canonicalLS} we briefly outline the history of results on
$1$-factorisations relevant to us.
A \emph{perfect} $1$-factorisation is a $1$-factorisation
in which each pair of $1$-factors forms a Hamiltonian
cycle. A famous conjecture by Kotzig asserts that there exists a
perfect $1$-factorisation of $K_{n}$ for all even $n$. Only three
infinite families of perfect $1$-factorisation of $K_{n}$ are known
\cite{BMW06} covering every $n$ of the form $n=2p$ or $n=p+1$ where
$p$ is an odd prime. Enumerations have been completed up to $K_{16}$
and constructions have been found for other sporadic orders;
see \cite{GW20} for the most recent list.

In \cite{DDS05} a weakened form of perfect $1$-factorisations is
studied, called sequentially perfect $1$-factorisations, for which the
factors of the $1$-factorisation are ordered and only (cyclically)
consecutive pairs of $1$-factors need to form a Hamiltonian cycle. It
is shown in~\cite{DDS05} that a sequentially perfect $1$-factorisation
of $K_{n}$ exists for all even $n$.

For a proper $k$-edge-colouring $\gamma$ of a graph $G$ with chromatic
index $k$, let $L_\gamma(k,G)$ be the maximum length of a bicoloured
cycle. Let $L(k,G)=\min_{\gamma}L_\gamma(k,G)$, where the minimum is
taken over proper $k$-edge-colourings of $G$. Dukes and Ling \cite{DL09}
gave upper bounds on $L(k,G)$ for two graphs of particular interest to
us, namely the complete bipartite graph and the complete graph. Benson
and Dukes \cite{BD16} subsequently improved the bound for $K_{n,n}$,
leading to the following:

\begin{theorem}\label{t:nolargecycles}
  For all $n$, $L(n,K_{n,n}) \leq 182$  and
  for all odd $n$, $L(n,K_{n+1}) \leq 1720$.
\end{theorem}
As $L(n,K_{n,n})$ is the minimum length of the longest row cycle among
Latin squares of order $n$, \tref{t:nolargecycles} shows that row
cycles longer than a fixed length can be avoided in Latin squares of
any order.

The Cayley tables of elementary abelian 2-groups provide a family of
Latin squares of unbounded order in which no row cycle has length more
than 2. More generally, by \tref{t:nolargecycles} it is known that the
minimum value of $\ell(L)$ is uniformly bounded for all orders $n$.
Thus the best bound that \tref{t:timetoextendpartiallabelling} can provide
on the worst case time complexity is $n^{O(1)+\log_2n}$.

In \cite{Mes09} it is shown that for all even $n$ and even $k<n$ there
exists a $1$-factorisation of $K_n$ avoiding a $k$-cycle (in the sense
that there is no pair of $1$-factors whose union contains a $k$-cycle).
Partial results for $1$-factorisation of $K_n$ avoiding $n$-cycles
(i.e., Hamiltonian cycles) are also obtained. It is also shown in
\cite{Mes09} that there exists a $1$-factorisation of $K_n$ avoiding
cycles of lengths 4 and 6 so long as $n\geq 10$ and $n\equiv2\bmod4$.

\subsection{Canonical $1$-factorisations}

In \sref{s:STS} we extended the algorithm defined in
\sref{s:canonicalLS} in two different ways to obtain canonical labelling
algorithms for Steiner quasigroups (or equivalently, for
Steiner triple systems) that preserve the structure of
the Steiner quasigroup. In this subsection we will perform the
analogous task for $1$-factorisations of the complete graph $K_{2n}$.

If we use the idempotent symmetric encoding of a $1$-factorisation
then we have an analogue of \tref{t:lift}.

\begin{theorem}\label{t:lift1F}
  Let $L$ be an idempotent symmetric Latin square and let
  $(\alpha,\beta,\gamma)$ be the canonical labelling of $L$ obtained
  by \aref{a:canonical}. Then $L_\gamma$ serves as a canonical
  representative of the rrs-isotopism class of $L$.
\end{theorem}

However, the idempotent form corresponds to rooted $1$-factorisations
of $K_n$, where one vertex is distinguished as a root, and cannot be
moved by any isomorphism.

From this point forward we will refer to an arbitrary
$1$-factorisation $F$ of the complete graph and the corresponding
unipotent Latin square $\UU = \UU(F)$. It is clear that we could use
the algorithm defined in \sref{s:canonicalLS} to obtain a canonical
labelling of the Latin square $\UU$. However this canonical labelling
may not preserve the unipotent structure nor the symmetry of a Latin
square.

We first consider the structure of the Latin square $\UU$ derived from the
$1$-factorisation $F$ of the complete graph $K_{2n}$. Let $(i,j)$ be indices of
two rows of $\UU$. If we consider the cycle structure of $r_i \cup r_j$ then we
observe that the unipotent elements $i \circ i = 1$ and $j \circ j = 1$ of $r_i
\cup r_j$ are contained in a $2$-cycle as $i \circ j = j \circ i = k$, for some
$k \neq 1$. Let $\alpha$ and $\gamma$ be two permutations of $[n]$ such that
$\gamma(1) = 1$, it is easy to see that the labelling $(\alpha, \alpha,
\gamma)$ preserves the unipotent structure and symmetry of the Latin square
$\UU$. This informs a small set of changes to the procedures set out in
\sref{s:canonicalLS}. 
A symmetric unipotent Latin
square $\UU$ is in \emph{factor standard form} if 
\begin{itemize}
\item $\perm{1}{2}(\UU)$ contains the $2$-cycle $(12)$,
\item the remaining row cycles in the first two rows
  are in standard form and sorted by
  length, with shorter cycles in columns with higher indices,
\item $\Gamma_{1,2}(\UU)$ is lexicographically maximum among the
  cycle structures of the pairs of rows in $\UU$.
\end{itemize}
Let $\FF$ be the set of all unipotent symmetric Latin squares that
are in factor standard form.

For an arbitrary $1$-factorisation $F$, let $\FF(\UU)$ be the set of
all labellings $\UU_{\alpha\alpha\gamma}$ that can be generated in the
following way. Let $R_{\text{max}}=R_{\text{max}}(\UU)$ be the set of
ordered pairs $(i,j)$ such that $\Gamma_{i,j}(\UU)$ is
lexicographically maximum. For each $(i,j) \in R_{\text{max}}$,
we start with the partial labelling $(\alpha,\alpha,\gamma)$
whose only defined labels are
$\alpha(i)=1=\gamma(1)$, and $\alpha(j) = 2=\gamma(i\circ j)$, giving
the $2$-cycle $(12)$ as the first row cycle of
$\UU_{\alpha\alpha\gamma}$. The other values of the permutation
$\alpha$ are then assigned in each possible way such
that the remaining row cycles in the
first two rows of $\UU_{\alpha\alpha\iota}$ are contiguous and weakly
decreasing in order of length. For each choice of $\alpha$, we
let $\gamma$ be the unique permutation
such that $\UU_{\alpha\alpha\gamma}$ is in factor standard form. The
following lemma follows immediately.

\begin{lemma}\label{l:rrs}
  If $\II(\UU)$ denotes the rrs-isotopism class of a Latin
  square $\UU$, then $\FF(\UU) = \FF \cap \II(\UU)$. 
\end{lemma}

Note that \lref{l:rrs} fails if rrs-isotopism is replaced by
isotopism.  There are unipotent symmetric Latin squares that are
isotopic but not rrs-isotopic, although curiously it was shown in
\cite{MW22} that this is only possible for orders that are divisible
by $4$.

We now give a small set of adjustments to the algorithm from
\sref{s:canonicalLS}, in order to compute a canonical labelling for
unipotent symmetric Latin squares. This gives an algorithm for finding
the canonical labelling of a $1$-factorisation that is more natural
than a direct application of the algorithm given in
\sref{s:canonicalLS}.

The only major changes are to Lines~\ref{line:updateP} and
\ref{line:firstbranch} of \aref{a:canonical}.  Before updating $P$,
define the partial labelling $(\alpha',\alpha',\gamma')$ such that
$\alpha'(i) = 1=\gamma'(1)$, $\alpha'(j) = 2=\gamma'(i\circ j)$ and
all other elements are unlabelled.  When $P$ is updated in
\Lref{line:updateP}, this should be done by calculating the cycle
structure but ignoring the row cycle of length $2$ on the idempotent
elements, and incrementing the value $P[k]$ by $2$ for each $k\in[n]$. 
Finally the call to \aref{a:branch} in \Lref{line:firstbranch} should
be replaced with $(\alpha^*,\alpha^*,\gamma^*) \gets$
\Call{Branch}{$L$, $(i,j)$, $(\alpha',\alpha',\gamma')$}.

By initialising the partial labelling $(\alpha',\alpha',\gamma')$ in
this way, any subsequent partial labelling formed by calling
\aref{a:rowcyclelabel} preserves symmetry, since the column index $b$
in \Lref{line:updatebeta} and the row index $a$ in
\Lref{line:updatealpha} must in fact be the same element, given that
$c_1=i$.  Therefore all isotopisms generated by \aref{a:canonical} are
of the form $(\alpha,\alpha,\gamma)$. Given that the initial Latin
square is unipotent and symmetric, its canonical form must also be
unipotent and symmetric.  Thus the following analogue to
\tref{t:isoLsamecan} is immediate.

\begin{theorem}\label{t:isoFsamecan}
  Let $L_1$ and $L_2$ be two rrs-isotopic unipotent symmetric Latin
  squares and let $L_1'$ and $L_2'$ be the Latin squares obtained by
  running the modified \aref{a:canonical} on $L_1$ and $L_2$,
  respectively. Then $L_1'=L_2'$.
\end{theorem}

Finally as with the Steiner quasigroup case, the worst case running
time is unaffected by the changes and we have:

\begin{theorem}\label{t:factcanonical}
  Let $L$ be a unipotent symmetric Latin square and let $\ell(L)$ be
  the length of the longest row cycle in $L$. \aref{a:canonical}
  determines the canonical labelling of $L$ in
  $O(n^{\log_2{n}-\log_2{\ell(L)} + 3})$ time.
\end{theorem}

Similarly to the Steiner quasigroup case, analogues to
Theorems~\ref{t:largeSS} and~\ref{t:longcyc} for $1$-factorisations
would allow us to determine the average-case time complexity of
\aref{a:canonical}.

\begin{conjecture}\label{con:1fact}
  Let $\UU$ be the Latin square representing a $1$-factorisation of
  $K_{2n}$ selected uniformly at random.  With probability $1-o(1)$,
  the length $\ell(\UU)$ of the longest row cycle in $\UU$ exceeds the
  size of the largest proper sub-$1$-factorisation of $\UU$.
\end{conjecture}

\begin{theorem}
  If \cref{con:1fact} is true, then canonical labelling (and hence
  also isomorphism testing) of $1$-factorisations of complete graphs,
  can be achieved in average-case $O(n^5)$ time.
\end{theorem}

We remark that although \cref{con:1fact} was stated for unipotent symmetric
Latin squares, we also believe it to be true for idempotent symmetric
Latin squares.

\section{Other combinatorial objects}\label{s:losotros}

In this section we briefly discuss applications of our algorithm to
isomorphism testing for other combinatorial objects that can be
encoded using Latin squares.

The isotopism problem for quasigroups is equivalent to the isotopism
problem for Latin squares as studied in \sref{s:canonicalLS}.
However, quasigroup isomorphism is in some ways more similar to the
isomorphism problem for Steiner triple systems that we studied in
\sref{s:STS}. To begin, we write the Cayley table of our quasigroup as
a Latin square $L$.  We find the set of pairs of rows of $L$ with the
lexicographically largest cycle structure.  For each such pair
$(i,j)$, we then consider isomorphisms which label the row cycles in
that pair of rows in a similar way to \aref{a:rowcyclelabel}. We
use the same list $P$ which pre-allocates columns according to
the lengths of cycles. Then when labelling a cycle of length $\ell$ we
will map its column indices to $P[\ell],\dots,P[\ell]+\ell-1$. We try
all $\ell$ options for which of its columns is put first (i.e.~in
position $P[\ell]$). Each subsequent column is determined by the rule
that after labelling column $c$ we label the column $c'$ that satisfies
$i\circ c'=j\circ c$.
In this way, we produce contiguous row cycles that are arranged in
weakly decreasing order of length. We do not have any control over the
symbols in these cycles, since we must apply the same permutation to
the symbols that we apply to the columns (and rows).  We first label
the row cycle that hits column $i$, then the row cycle that hits
column $j$ (if that is a different cycle). After that, we label row cycles
in an order decided by the first $\star$ that we find in
the array induced by the labelled rows and labelled columns. This
array is traversed with the same $\Succ$ function that we used
in \sref{s:canonicalLS}. When we encounter a $\star$ its preimage
will be in the next row cycle that we label. If we complete a
traversal without finding a $\star$, then the labelled rows
and columns must induced a subquasigroup. In such a situation we branch
by labelling a longest unlabelled row cycle. 

Our method will run in polynomial time for any family of quasigroups
in which no longest row cycle lies inside a chain of nested
subquasigroups where the length of the chain is unbounded (hence
requiring an unbounded number of branch steps). In particular, this is
true of almost all quasigroups, since by \tref{t:largeSS} and
\tref{t:longcyc} they will contain a row cycle longer than the order
of their largest subquasigroup. We thus have:

\begin{theorem}
  Quasigroup isomorphism can be solved in average-case polynomial time.
\end{theorem}

Colbourn and Colbourn \cite{CC80} noted that in polynomial time
their algorithm canonically labelled quasigroups who 
largest proper subquasigroup has bounded order.

A \emph{loop} is a quasigroup with an identity element. The Cayley
table of a loop is a Latin square with one row and one column in which
the symbols occur in order. Every quasigroup is isotopic to a loop.
Indeed, isotopism provides an $n!(n-1)!$ to 1 map from quasigroups to loops
which preserves the size of the largest subsquare and the length of the
longest row cycle. Hence, we know that almost all loops have a row cycle
longer than the order of their largest subloop, which gives:

\begin{theorem}
  Loop isomorphism can be solved in average-case polynomial time.
\end{theorem}

The worst case performance for our quasigroup isomorphism algorithm
remains $n^{O(\log n)}$, matching the original bound by Miller
\cite{Mil78}.  The algorithm can of course be utilised for any special
subclass, such as groups or the quasigroups which encode Mendelsohn triple
systems. However, doing so will lose the guarantee of polynomial
average-case performance.

Miller \cite{Mil78} also described how an algorithm for 
labelling Latin squares can be used for testing isomorphism of nets and
finite projective planes and affine planes. In each case, the complexity is
at most some polynomial times the complexity of the Latin square labelling.
Hence, our algorithm from \sref{s:canonicalLS} can be used to
canonically label nets, affine planes and projective planes. If any of
the Latin squares used to build such an object has a long cycle
then our algorithm will run in polynomial time. However, we no longer
have a provable bound on the average-case complexity. The typical structure
of Latin squares that occur in affine planes is plausibly very different
to the structure of random Latin squares. At the very least it is not
reasonable to expect information about average cases when even the existence
question is far from settled.

\section{Discussion}\label{s:con}

In this paper we have introduced an algorithm for canonically labelling Latin
squares, and any other combinatorial object that can be encoded using Latin
squares.  In a first, we have been able to prove average-case polynomial
time-complexity for our algorithm.  When finding a canonical representative,
our algorithm will grow a partial labelling until it is a labelling. In the
process it sometimes makes choices that lead to branching in the search tree.
These choices are made every time \aref{a:branch} is called, which happens when
the partial labelling induces a Latin square.  As a consequence, a long chain
of nested subsquares can create sufficiently many branches that the complexity
becomes superpolynomial. Fortunately, such chains are rare in a relative sense.
However, the following  construction demonstrates that there are many species
of Latin squares that contain them.

Start with the Cayley table of the elementary abelian $2$-group of order $n$.
It can be tiled with a set $I$ of $n^2/4$ intercalates. It also has the
property that any sequence of choices for the labelling of a subsquare will
lead to a chain of nested subsquares of length $\log_2(n)$.  Now form a set $S$
of $2^{n^2/4}$ distinct Latin squares by turning any subset of the intercalates
in $I$ (to turn an intercalate, replace it by the other possible intercalate on
the same symbols). These squares all have the property that any sequence of
choices for the labelling of any subsquare will lead to a chain of nested
subsquares of length either $\log_2(n)$ or $\log_2(n)-1$. Also there are a lot
of different species represented in $S$.  Any species contains at most
$6(n!)^3=2^{O(n\log n)}$ different Latin squares, so there are at least
$2^{n^2(1/4-o(1))}$ species represented in $S$.

As already discussed, our algorithm sometimes performs poorly when given a Latin
square containing a long chain of nested subsquares. However, unlike the
algorithm of Colbourn and Colbourn \cite{CC80}, the existence of such a chain
does not necessarily imply poor performance. A Latin square may contain a
long chain of nested subsquares and also contain a long cycle, perhaps even
one of length $n$. As we have
shown, the existence of a long cycle gives polynomial time complexity.
For example, it is easy to construct a Latin
square that contains the Cayley table of an elementary abelian $2$-group as a
subsquare in its first $n/2$ rows, but has a row $r_{n/2+1}$ such that
$r_1 \cup r_{n/2+1}$ is a row cycle of length $n$.  Such a Latin square
contains a chain of nested subsquares of length $\log_2(n)$, but will be
labelled by our algorithm in polynomial time.

\begin{table}
\begin{center}
\begin{tabular}{|c|ccccc|}
\hline
Order&Min&Max&Mode&Mean&Std Dev.\\
\hline
10&0&23&13&12.2&3.05\\
15&3&51&18&19.0&3.98\\
20&0&48&26&25.8&4.74\\
25&9&68&32&32.6&5.40\\
30&10&68&39&39.4&5.99\\
35&19&92&45&46.2&6.54\\
40&16&89&53&53.0&7.03\\
45&24&102&60&59.8&7.50\\
50&28&110&66&66.6&7.94\\
\hline
\end{tabular}
\caption{\label{T:cyclendata}Number $H$ of Hamiltonian row cycles in
random Latin squares.}
\end{center}
\end{table}

On the basis of \cite{CGW08} it is reasonable to expect that the
probability of two given rows of a random $n\times n$ Latin square
producing a Hamiltonian cycle is asymptotically $e/n$ (which is the
asymptotic probability that a random derangement of $[n]$ will be an
$n$-cycle). Since there are $\binom{n}{2}$ pairs of rows, we might
then anticipate that the expected value of $H$, the number of
Hamiltonian row cycles in the whole square, would be asymptotic to
$e(n-1)/2$. To test this hypothesis we used the Jacobson-Matthews
Markov chain \cite{JM96} to generate one million Latin squares of each
of the orders 10, 15, 20, 25, 30, 35, 40, 45 and 50.  The resulting
values of $H$ are summarised in \Tref{T:cyclendata}, where we give the
minimum value, maximum value, mode, mean and standard deviation.  In
each case the mean value of $H$ agreed exactly with $e(n-1)/2$ to the
3 significant digit accuracy quoted. Also, in our sample there was
only one square of order greater than 10 that had $H=0$, and it had
order $20$. The data in \Tref{T:cyclendata} displays some bias towards
a higher maximum value of $H$ for odd orders compared to even
orders. This can be explained by an observation based on the parity of
row permutations within a set of three rows. If the order is odd then
all three pairs of rows within a set of three rows can produce
Hamiltonian cycles, whereas this is not possible when the order is
even. This effect will diminish for larger orders, where the probability
of more than one Hamiltonian cycle within any given set of three rows becomes
negligible.

\begin{figure}
  \centering
  \begin{tikzpicture}
    \begin{axis}
      [
        mark size = 1.0pt,
        scaled ticks = false,
        width = 0.7\textwidth,
        height = 0.5\textwidth,
      xmax = 115,
      xmin = 20,
      yticklabel style={/pgf/number format/1000 sep=},
      ymax = 52000,
      ymin = 0
      ]
      \addplot[
        color = black,
        fill = white,
        mark = square,
        only marks
        ]
        coordinates {
        ( 100, 17 )
        ( 101, 13 )
        ( 102, 6 )
        ( 103, 5 )
        ( 104, 2 )
        ( 106, 1 )
        ( 110, 1 )
        ( 28, 1 )
        ( 30, 1 )
        ( 31, 1 )
        ( 32, 3 )
        ( 33, 2 )
        ( 34, 1 )
        ( 35, 5 )
        ( 36, 16 )
        ( 37, 26 )
        ( 38, 64 )
        ( 39, 87 )
        ( 40, 134 )
        ( 41, 204 )
        ( 42, 326 )
        ( 43, 477 )
        ( 44, 698 )
        ( 45, 1059 )
        ( 46, 1475 )
        ( 47, 2125 )
        ( 48, 2921 )
        ( 49, 3989 )
        ( 50, 5426 )
        ( 51, 7221 )
        ( 52, 9068 )
        ( 53, 11502 )
        ( 54, 14528 )
        ( 55, 17604 )
        ( 56, 20985 )
        ( 57, 24839 )
        ( 58, 28627 )
        ( 59, 32704 )
        ( 60, 36510 )
        ( 61, 40249 )
        ( 62, 43687 )
        ( 63, 46340 )
        ( 64, 47873 )
        ( 65, 49571 )
        ( 66, 50254 )
        ( 67, 49961 )
        ( 68, 49338 )
        ( 69, 47419 )
        ( 70, 44991 )
        ( 71, 42313 )
        ( 72, 38970 )
        ( 73, 35482 )
        ( 74, 31676 )
        ( 75, 28153 )
        ( 76, 24255 )
        ( 77, 20674 )
        ( 78, 17524 )
        ( 79, 14530 )
        ( 80, 12088 )
        ( 81, 9775 )
        ( 82, 7827 )
        ( 83, 6100 )
        ( 84, 4695 )
        ( 85, 3509 )
        ( 86, 2754 )
        ( 87, 2143 )
        ( 88, 1493 )
        ( 89, 1091 )
        ( 90, 842 )
        ( 91, 554 )
        ( 92, 383 )
        ( 93, 290 )
        ( 94, 188 )
        ( 95, 112 )
        ( 96, 106 )
        ( 97, 62 )
        ( 98, 36 )
        ( 99, 18 )        
      };
    \end{axis}
  \end{tikzpicture}
  \caption{\label{f:plot50}Frequency plot of $H$ among random Latin squares of
  order $50$.}
\end{figure}

\fref{f:plot50} shows a plot of the results for our one million Latin
squares of order 50. The horizontal axis shows $H$ and the vertical
axis shows how many of the million squares achieved each $H$ value.
Similar behaviour was observed for the other orders.

On the basis of our numerical evidence, we advance this conjecture:

\begin{conjecture}
  Let $\ell(L)$ denote the length of the longest row cycle in a Latin
  square $L$. As $n\rightarrow\infty$, asymptotically almost all
  Latin squares of order $n$ have $\ell(L)=n$.
\end{conjecture}

Furthermore, it seems likely that there will typically be a linear
number of Hamiltonian cycles amongst the pairs of rows of $L$. If this
is true then the algorithm we have described will have $O(n^4)$
average-case time complexity, improving on \tref{t:polycanonical} by a
factor of $n$. In practice, we do not need to use the longest
cycles. There could be good reasons to use shorter cycles if there is
a cycle length which occurs but is uncommon.  When finding the lengths
of all cycles in a Latin square, it is a simple matter to count how
many cycles of each length occur. By choosing a length that is
uncommon but still long enough to exceed the order of any subsquare,
we might obtain an algorithm that has $O(n^3)$ average-case time
complexity. We could not do better with this approach because it takes
cubic time to find all the cycle lengths. However, to prove the
$O(n^3)$ bound we would need to show that with high probability there
is some cycle length which occurs but does not occur many times. This
seems plausible but beyond present methods of proof.

There is another adjustment to our algorithm which does not affect our
analysis but is likely to improve the practical running time. For
simplicity, we have concentrated throughout on row cycles, but it is
well known that Latin squares also possess column cycles and symbol
cycles. By running on a conjugate Latin square if necessary, our
algorithm could be adjusted to use whichever is most favourable out of
the row, column and symbol cycles. The lengths of these different
types of cycles can vary considerably.  It is possible for a Latin
square to have no Hamiltonian row cycles while the row cycles in one
of its conjugates are all Hamiltonian \cite{AW22}.

\subsection*{Acknowledgements}

The authors are very grateful for feedback from Matthew Kwan, Brendan
McKay, G\'abor Nagy, Adolfo Piperno, Pascal Schweitzer and Petr
Vojt\v{e}chovsk\'y.

 
  \let\oldthebibliography=\thebibliography
  \let\endoldthebibliography=\endthebibliography
  \renewenvironment{thebibliography}[1]{%
    \begin{oldthebibliography}{#1}%
      \setlength{\parskip}{0.2ex}%
      \setlength{\itemsep}{0.2ex}%
  }%
  {%
    \end{oldthebibliography}%
  }

\end{document}